\documentclass[a4paper,12pt]{article}
\usepackage[latin1]{inputenc}
\usepackage{amsthm}
\usepackage{amsfonts}
\usepackage[dvips]{graphicx}
\usepackage{latexsym}
\usepackage{amsmath}
\usepackage{color}

\newtheorem{lem}{Lemma}
\newtheorem{thm}{Theorem}

\newtheorem{prop}{Proposition}

\newtheorem{oss}{Remark}

\allowdisplaybreaks[4]

\begin{document}

\title{\large\textbf{STRONG SOLUTIONS FOR TWO-DIMENSIONAL \\ NONLOCAL CAHN-HILLIARD-NAVIER-STOKES SYSTEMS}}

\author{
{\sc Sergio Frigeri}\\
Dipartimento di Matematica {\it F. Enriques}
\\Universit\`{a} degli Studi di Milano\\Milano I-20133, Italy\\
\textit{sergio.frigeri@unimi.it}
\\
\\
{\sc Maurizio Grasselli}\\
Dipartimento di Matematica {\it F. Brioschi}\\
Politecnico di Milano\\
Milano I-20133, Italy \\
\textit{maurizio.grasselli@polimi.it}
\\
\\
{\sc Pavel Krej\v{c}\'{\i}}\\
Czech Academy of Sciences\\
CZ-11567 Praha 1, Czech Republic\\
\textit{krejci@math.cas.cz}}

\maketitle

\begin{abstract}\noindent
A well-known diffuse interface model for incompressible isothermal mixtures of two immiscible fluids
consists of the Navier-Stokes system coupled with a convective Cahn-Hilliard equation. In some recent contributions
the standard Cahn-Hilliard equation has been replaced by its nonlocal version.
The corresponding system
is physically more relevant and mathematically more challenging. Indeed, the only known results are essentially the existence of a global weak solution and the existence of a suitable notion of global attractor for the corresponding dynamical system defined without uniqueness.
In fact, even in the two-dimensional case, uniqueness of weak solutions is still an open problem.
Here we take a step forward in the case of regular potentials. First we prove the existence of a (unique) strong solution in two dimensions.
Then we show that any weak solution regularizes in finite time uniformly with respect to bounded sets of initial data.
This result allows us to deduce that the global attractor is the union of all the bounded complete trajectories which are strong solutions. We also demonstrate that each trajectory converges to a single equilibrium, provided that the potential is real analytic and the external forces vanish.
\\
\\
\noindent \textbf{Keywords}: Navier-Stokes equations, nonlocal
Cahn-Hilliard equations, regular potentials, incompressible binary
fluids, strong solutions, global attractors, convergence to equilibrium,
{\L}ojasiewicz-Simon inequality.
\\
\\
\textbf{AMS Subject Classification 2010}: 35Q30, 37L30, 45K05, 76D03, 76T99.
\end{abstract}

\section{Introduction}\setcounter{equation}{0}
The evolution of an incompressible mixture of two immiscible fluids can be described
through a diffuse interface model (cf., e.g., \cite{GPV,HMR,JV,M} and their references).
Assuming that the temperature variations are negligible, taking the density is equal to one, and suppose
the viscosity $\nu$ to be constant, the model H (see \cite{HH}) reduces to the so-called  Cahn-Hilliard-Navier-Stokes system
\begin{align*}
& \varphi_t+u\cdot\nabla\varphi=\nabla\cdot (\kappa\nabla\mu), \\
&\mu=-\Delta\varphi+F^\prime(\varphi),  \\
&u_t-\nu\Delta u+(u\cdot\nabla)u+\nabla\pi=\mu\nabla\varphi+h(t),
 \\
&\mbox{div}(u)=0,
\end{align*}
in $\Omega\times (0,\infty)$, where $\Omega \subset
\mathbb{R}^d$, $d=2,3$, is a bounded domain.
Here $u$ denotes the (average) velocity and $\varphi$ is the difference of the two fluid concentrations.
Moreover, $\kappa>0$ is the mobility coefficient, $F$ is a suitable double well potential density, $\pi$ the pressure and
$h$ a given external (non-gradient) force.

The existing theoretical literature (see, for instance, \cite{A1,A2,B,GG1,GG2,GG3,S,ZWH}) can be summarized by saying that all the results known
for the Navier-Stokes system can be extended to the Cahn-Hilliard-Navier-Stokes one, with some additional technical difficulties when, for instance,
$F$ is a singular (i.e. logarithmic) potential and/or the mobility $\kappa$ depends on $\varphi$ and vanishes at pure phases (cf. \cite{A1,B}).
However, we recall that the Cahn-Hilliard equation has a phenomenological nature (cf. \cite{CH}). Instead, a rigorous derivation from a microscopic model  yields a nonlocal equation (see \cite{GL1,GL2}). In this case the chemical potential $\mu$ has the following form
$$
\mu=a\varphi-J\ast\varphi+F^\prime(\varphi),
$$
where $*$ denotes the convolution product over $\Omega$, $J:\mathbb{R}^d \to \mathbb{R}$ is a sufficiently smooth interaction kernel
such that $J(x)=J(-x)$ and $a(x)=\displaystyle\int_\Omega k(x-y)dy$. Motivated by this fact, in \cite{CFG} we have introduced and
analyzed the following nonlocal Cahn-Hilliard-Navier-Stokes system
\begin{align}
& \varphi_t+u\cdot\nabla\varphi=\Delta\mu,\label{sy1}\\
&\mu=a\varphi-J\ast\varphi+F^\prime(\varphi), \label{sy2}\\
&u_t-\nu\Delta u+(u\cdot\nabla)u+\nabla\pi=\mu\nabla\varphi+h(t),
\label{sy3}\\
&\mbox{div}(u)=0, \label{sy4}
\end{align}
endowed with boundary and initial conditions
\begin{align}
&\frac{\partial\mu}{\partial n}=0,\quad u=0\quad\mbox{on }
\partial\Omega\times (0,T)\label{sy5}\\
&u(0)=u_0,\quad\varphi(0)=\varphi_0\quad\mbox{in }\Omega.
\label{sy6}
\end{align}
For such a problem we have proven first the existence of a global weak solution satisfying an energy inequality (equality in dimension two) for a
regular potential $F$ (see \cite{CFG}).
Then in \cite{FG1} we have established the existence of a global attractor for the generalized semiflow ($d=2$) and a trajectory attractor ($d=3$).
Similar results have recently been extended to singular potentials of logarithmic type (cf. \cite{FG2}). However, an important issue has been left open: the uniqueness
of weak solutions in dimension two. This is well known for the standard local models and it
suggests that the present model is more difficult to handle. The main reason
seems to be the poorer regularity of $\varphi$ which makes the capillarity term (i.e. the Korteweg force) $\mu\nabla\varphi$ difficult to handle (see \cite{CFG}). Here
we are not able to address this issue but we come close. More precisely, we prove the existence of a (unique) strong solution
and the regularization in finite time of any weak solution. The latter is uniform with respect to bounded set of initial data so that, as a by-product, we deduce that the global attractor we found in \cite{FG1} is smooth. More precisely, it is the union of all the bounded complete trajectories which are strong solutions to \eqref{sy1}-\eqref{sy6}. Finally, taking advantage of the regularization property, we show that any weak trajectory does converge to a unique equilibrium (cf. \cite{GG,LP,LP2} for nonlocal Cahn-Hilliard equations).
\section{Notation and known results}\setcounter{equation}{0}
We set $H:=L^2(\Omega)$ and
$V:=H^1(\Omega)$.
For every $f\in V'$ we denote by $\overline{f}$
the average of $f$ over $\Omega$, i.e., $\overline{f}:=|\Omega|^{-1}\langle f,1\rangle$.
Here $|\Omega|$ is the Lebesgue measure of $\Omega$. We assume that $\partial\Omega$
is smooth enough.

Then we introduce the Hilbert spaces
$$V_0:=\{v\in V:\overline{v}=0\},\qquad V_0':=\{f\in V':\overline{f}=0\},$$
and the operator $A:V\to V'$, $A\in\mathcal{L}(V,V')$, defined by
$$\langle Au,v\rangle:=\int_{\Omega}\nabla u\cdot\nabla v\qquad\forall u,v\in V.$$
We recall that $A$ maps $V$ onto $V_0'$ and the restriction of $A$ to $V_0$ maps $V_0$ onto $V_0'$
isomorphically. Further, we denote by $\mathcal{N}:V_0'\to V_0$ the inverse map defined by
$$A\mathcal{N}f=f,\quad\forall f\in V_0'\qquad\mbox{and}\qquad\mathcal{N}Au=u,\quad\forall u\in V_0.$$
As is well known, for every $f\in V_0'$, $\mathcal{N}f$ is the unique
solution with zero mean value of the Neumann problem
\begin{equation*}
\left\{\begin{array}{ll}
-\Delta u=f,\qquad\mbox{in }\Omega\\
\frac{\partial u}{\partial n}=0,\qquad\mbox{on }\partial\Omega.
\end{array}\right.
\end{equation*}
In addition, we have
\begin{align}
&\langle Au,\mathcal{N}f\rangle=\langle f,u\rangle,\qquad\forall u\in V,\quad\forall f\in V_0',\\
&\langle f,\mathcal{N}g\rangle=\langle g,\mathcal{N}f\rangle=\int_{\Omega}\nabla(\mathcal{N}f)
\cdot\nabla(\mathcal{N}g),\qquad\forall f,g\in V_0'.
\end{align}
We consider the canonical Hilbert spaces for the Navier-Stokes
equations with no-slip boundary condition (see, e.g., \cite{T})
$$
G_{div}:=\overline{\{u\in
C^\infty_0(\Omega)^d:\mbox{
div}(u)=0\}}^{L^2(\Omega)^d},\quad
V_{div}:=\{u\in H_0^1(\Omega)^d:\mbox{ div}(u)=0\}.
$$
We denote by $\|\cdot\|$ and $(\cdot,\cdot)$ the norm and the
scalar product on both $H$ and $G_{div}$, respectively. Instead, $V_{div}$ is
endowed with the scalar product
$$(u,v)_{V_{div}}=(\nabla u,\nabla v),\qquad\forall u,v\in V_{div}.$$

We shall also need to introduce the Stokes operator $S$
with no-slip boundary condition. More precisely, $S:D(S)\subset G_{div}\to G_{div}$ is defined as
$S:=-P\Delta$  with domain $D(S)=H^2(\Omega)^d\cap V_{div}$,
where $P:L^2(\Omega)^d\to G_{div}$ is the Leray projector. Notice that we have
$$(Su,v)=(u,v)_{V_{div}}=(\nabla u,\nabla v),\qquad\forall u\in D(S),\quad\forall v\in V_{div},$$
and $S^{-1}:G_{div}\to G_{div}$ is a self-adjoint compact operator in $G_{div}$.
Thus, according with classical results,
$S$ possesses a sequence of eigenvalues $\{\lambda_j\}$ with $0<\lambda_1\leq\lambda_2\leq\cdots$ and $\lambda_j\to\infty$,
and a family $\{w_j\}\subset D(S)$ of eigenfunctions which is orthonormal in $G_{div}$. Let us also recall
Poincar\'{e}'s inequality
$$\lambda_1\Vert u\Vert^2\leq\Vert\nabla u\Vert^2,\qquad\forall u\in V_{div}.$$

The trilinear form $b$ which appears in the weak formulation of the
Navier-Stokes equations is defined as follows
$$b(u,v,w)=\int_{\Omega}(u\cdot\nabla)v\cdot w,\qquad\forall u,v,w\in V_{div},$$
and the associated bilinear operator $B$ from $V_{div}\times V_{div}$ into $V_{div}'$ is defined by
\begin{align}
& \langle B(u,v),w\rangle:=b(u,v,w),\qquad\forall u,v,w\in V_{div}.\nonumber
\end{align}
We shall set $B(u,u):=Bu$, for all $u\in V_{div}$.
We recall that we have
\begin{equation}
b(u,w,v)=-b(u,v,w),\qquad\forall u,v,w\in V_{div},\label{ftril1}
\end{equation}
and that the following estimates hold in dimension two
\begin{align}
&|b(u,v,w)|\leq c\|u\|^{1/2}\|\nabla u\|^{1/2}\|\nabla v\|\|w\|^{1/2}\|\nabla w\|^{1/2},\quad\forall u,v,w\in V_{div},\label{ftril2}\\
&|b(u,v,w)|\leq c\|u\|^{1/2}\|\nabla u\|^{1/2}\|\nabla v\|^{1/2}\|Sv\|^{1/2}\|w\|,\quad\forall u\in V_{div},\; v\in D(S),\; w\in G_{div}.\label{ftril3}
\end{align}

If $X$ is a Banach space and $\tau\in\mathbb{R}$, we shall denote by $L^p_{tb}(\tau,\infty;X)$, $1\leq p<\infty$,
the space of functions $f\in L^p_{loc}([\tau,\infty);X)$ that are translation bounded in $L^p_{loc}([\tau,\infty);X)$,
i.e. such that
\begin{align}
&\Vert f\Vert_{L^p_{tb}(\tau,\infty;X)}^p:=\sup_{t\geq\tau}\int_t^{t+1}\Vert f(s)\Vert_X^p ds<\infty.
\end{align}

We shall use the following lemma. Its simple proof is given below for the reader's convenience.
\begin{lem}
Let $f\in L^{p_1}(\tau,\infty;X)$ with $f_t\in
L^{p_2}_{tb}(\tau,\infty;X)$, where $1\leq p_1<\infty$,
$1<p_2\leq \infty$, $\tau\in\mathbb{R}$ and $X$ is a reflexive
Banach space. Then $f(t)\to 0$ in $X$ as $t\to\infty$.
\label{prellem}
\end{lem}
\begin{proof}
We argue by contradiction. Suppose there exist a sequence $\{t_n\}$
with $t_n\to\infty$ and a constant $\sigma>0$ such that $\Vert
f(t_n)\Vert_X\geq\sigma$, for all $n$. Set $\tau_n:=t_n+1/n$.
Since $f\in L^{p_1}(\tau,\infty;X)$ with $1\leq p_1<\infty$,
then, by possibly extracting a subsequence, for every $n$ there
exists $t_n'\in[t_n,\tau_n]$ such that $\Vert
f(t_n')\Vert_X\leq\sigma/2$. We therefore get a contradiction,
since, denoting by $p_2'\in[1,\infty)$ the conjugate of $p_2$,
\begin{align}
&0<\frac{\sigma}{2}\leq\Vert
f(t_n')-f(t_n)\Vert_X\leq\int_{t_n}^{t_n'}\Vert f_t(s)\Vert_X ds
\leq\Vert
f_t\Vert_{L^{p_2}_{tb}(\tau,\infty;X)}\frac{1}{n^{p_2'}}\to
0.\nonumber
\end{align}
\end{proof}

We also report the uniform Gronwall lemma which will be useful in the sequel
(see, e.g., \cite{TInf}).
\begin{lem}
\label{unifGronw}
Let $\Phi$ be an absolutely continuous nonnegative function on $[\tau,\infty)$ and $\omega_1,\omega_2$ two nonnegative
locally summable functions on $[\tau,\infty)$ satisfying
\begin{align}
&\frac{d}{dt}\Phi(t)\leq\omega_1(t)\Phi(t)+\omega_2(t),\qquad\mbox{for a.e. }t\in[\tau,\infty),
\end{align}
and such that
\begin{align}
&\int_t^{t+1}\omega_i(s)ds\leq a_i,\quad i=1,2,\qquad\int_t^{t+1}\Phi(s)ds\leq a_3,
\end{align}
for all $t\geq\tau$, where $a_1,a_2,a_3$ are some nonnegative constants. Then
\begin{align}
&\Phi(t+1)\leq(a_2+a_3)e^{a_1},\qquad\forall t\geq\tau.
\end{align}
\end{lem}

We now summarize the main results of \cite{CFG}. They are concerned with the existence of dissipative weak solutions
and the validity of the energy identity and of a dissipative estimate in dimension two.

The assumptions on $J$ and $F$ are listed below

\begin{description}
 \item[(H1)] $J\in W^{1,1}(\mathbb{R}^d),\quad
     J(x)=J(-x),\quad a\geq 0\quad\mbox{a.e. in } \Omega$.
 \item[(H2)] $F\in C^{2,1}_{loc}(\mathbb{R})$ and there exists $c_0>0$
     such that
             $$F^{\prime\prime}(s)+a(x)\geq c_0,\qquad\forall s\in\mathbb{R},\quad\mbox{a.e. }x\in\Omega.$$
 \item[(H3)] $F\in C^2(\mathbb{R})$ and there exist $c_1>0$,
    $c_2>0$ and $q>0$ such that
            $$F^{\prime\prime}(s)+a(x)\geq c_1\vert s\vert^{2q} - c_2,
            \qquad\forall s\in\mathbb{R},\quad\mbox{a.e. }x\in\Omega.$$
 \item[(H4)] There exist $c_3>0$, $c_4\geq0$ and $r\in(1,2]$
     such that
             $$|F^\prime(s)|^r\leq c_3|F(s)|+c_4,\qquad
             \forall s\in\mathbb{R}.$$
\end{description}

\begin{oss}
{\upshape
Assumption $J\in W^{1,1}(\mathbb{R}^d)$ can be weakened. Indeed, it can be replaced by
$J\in W^{1,1}(B_{\delta})$, where $B_{\delta}:=\{z\in\mathbb{R}^d:|z|<\delta\}$ with $\delta:=\mbox{diam}(\Omega)$, or also
by (see, e.g., \cite{BH1})
$$\sup_{x\in\Omega}\int_{\Omega}\big(|J(x-y)|+|\nabla J(x-y)|\big)dy<\infty.$$
}
\end{oss}

The above assumptions allow to prove the following result (see \cite{CFG})

\begin{thm}
\label{thm} Let $h\in L^2_{loc}([0,\infty);V^\prime_{div})$,
$u_0\in G_{div}$, $\varphi_0\in H$ such that $F(\varphi_0)\in
L^1(\Omega)$ and suppose that (H1)-(H4) are satisfied. Then, for
every given $T>0$, there exists a weak solution $[u,\varphi]$ to \eqref{sy1}--\eqref{sy6}
such that
\begin{align}
&u\in L^{\infty}(0,T;G_{div})\cap L^2(0,T;V_{div}),\quad\varphi \in L^\infty(0,T;L^{2+2q}(\Omega))\cap L^2(0,T;V),\\
&u_t\in L^{4/3}(0,T;V_{div}'),\quad\varphi_t\in L^{4/3}(0,T;V'),\qquad d=3,\\
&u_t\in L^2(0,T;V_{div}'),\quad d=2,\\
&\varphi_t\in L^2(0,T;V'),\quad d=2\quad\mbox{ or }
 \quad d=3 \mbox{ and } q\geq 1/2,
\end{align}
and satisfying the energy inequality
\begin{equation}
\mathcal{E}(u(t),\varphi(t)) +\int_0^t\Big(\nu\|\nabla u \|^2+\|\nabla\mu \|^2\Big)d\tau
\leq\mathcal{E}(u_0,\varphi_0)+\int_0^t\langle h(\tau),u \rangle d\tau,\label{ei}
\end{equation}
for every $t>0$,
where we have set
$$\mathcal{E}(u(t),\varphi(t))=\frac{1}{2}\|u(t)\|^2+\frac{1}{4}
\int_{\Omega}\int_{\Omega}J(x-y)(\varphi(x,t)-\varphi(y,t))^2
dxdy+\int_{\Omega}F(\varphi(t)).$$
If $d=2$, then any weak solution satisfies the energy identity
\begin{equation}
\frac{d}{dt}\mathcal{E}(u,\varphi)
+\nu\|\nabla u \|^2+\|\nabla\mu\|^2=\langle h(t),u\rangle,
\label{eniden}
\end{equation}
In particular we have $u\in C([0,\infty);G_{div})$, $\varphi\in C([0,\infty);H)$ and $\int_\Omega F(\varphi)\in C([0,\infty))$.
Furthermore, if $d=2$ and $h\in L^2_{tb}(0,\infty;V_{div}')$,
then any weak solution satisfies
also the dissipative estimate
\begin{equation}
\mathcal{E}(u(t),\varphi(t))\leq \mathcal{E}(u_0,\varphi_0)e^{-kt}+ F(m_0)|\Omega| + K,
\qquad\forall t\geq 0,\label{dissest}
\end{equation}
where $m_0=(\varphi_0,1)$ and $k$, $K$ are two positive constants
which are independent of the initial data, with $K$ depending on
$\Omega$, $\nu$, $J$, $F$ and $\|h\|_{L^2_{tb}(0,\infty;V_{div}')}$.

\end{thm}

\begin{oss}
{\upshape All the previous results hold for a viscosity $\nu$ depending on $\phi$ which is sufficiently smooth and bounded from above and from below (see \cite{CFG}, cf. also \cite{FG1,FG2}). Here we assume $\nu$ to be constant just to avoid further technicalities in the sequel.}
\end{oss}


\section{Strong solutions in two dimensions}\setcounter{equation}{0}
In this section we state and prove our main result, namely the existence of a (global) strong solution to \eqref{sy1}--\eqref{sy6} and its uniqueness. More precisely, we have

\begin{thm}\label{main}
Let $h\in L^2_{loc}([0,\infty);G_{div})$,
$u_0\in V_{div}$, $\varphi_0\in V\cap L^{\infty}(\Omega)$
and suppose that (H1)-(H4) are satisfied. Then, for
every given $T>0$, there exists a weak solution $[u,\varphi]$
such that
\begin{align}
&u\in L^{\infty}(0,T;V_{div})\cap L^2(0,T;H^2(\Omega)^2),\quad\varphi \in L^\infty(\Omega\times(0,T))\cap L^\infty(0,T;V),\label{reg1}\\
&u_t\in L^2(0,T;G_{div}),\quad\varphi_t\in L^2(0,T;H).\label{reg2}
\end{align}
Furthermore, suppose in addition that $F\in C^3(\mathbb{R})$ and that $\varphi_0\in H^2(\Omega)$. Then, system \eqref{sy1}--\eqref{sy4}
admits a unique strong solution on $[0,T]$ satisfying \eqref{reg1}, \eqref{reg2} and also
\begin{align}
&\varphi\in L^\infty(0,T;W^{1,p}(\Omega)),\quad 2\leq p<\infty,\label{reg3}\\
&\varphi_t\in L^\infty(0,T;H)\cap L^2(0,T;V).\label{reg4}
\end{align}
If $J\in W^{2,1}(\mathbb{R}^2)$, we have in addition
\begin{align}
&\varphi\in L^\infty(0,T;H^2(\Omega)).\label{reg10}
\end{align}
Moreover, let $[u_{0i},\varphi_{0i},h_i]\in V_{div}\times H^2(\Omega)\times L^2_{loc}([0,\infty);G_{div})$, $i=1,2$,
be two sets of data and denote by $[u_i,\varphi_i]$ the corresponding solutions.
Then, there exists a positive constant $\Lambda$ which is a continuous and increasing function
of the norms
of the data of two solutions and which also depends on $T$, $F$, $J$, $\Omega$, $\nu$, such that the following continuous dependence estimate holds
\begin{align}
&\Vert u_2(t)-u_1(t)\Vert^2+\Vert\varphi_2(t)-\varphi_1(t)\Vert_{V_0'}^2\nonumber\\
&+\int_0^t\Vert\nabla u_2(\tau)-\nabla u_1(\tau)\Vert^2 d\tau
+\int_0^t\Vert\varphi_2(\tau)-\varphi_1(\tau)\Vert^2 d\tau\nonumber\\
&\leq\Lambda\Big(\Vert u_{02}-u_{01}\Vert^2
+\Vert\varphi_{02}-\varphi_{01}\Vert_{V_0'}^2+\Vert h_2-h_1\Vert_{L^2(0,T;G_{div})}^2\Big),
\label{contdipest}
\end{align}
for every $t\in[0,T]$.

\end{thm}

\begin{oss}
\label{timereg}
{\upshape
The regularity properties \eqref{reg1}--\eqref{reg10} imply that
$$u\in C([0,\infty);V_{div}),\quad\varphi\in C([0,T];V)\cap C_w([0,T];H^2(\Omega))$$
Actually, we have also $\varphi\in C([0,T];H^\delta(\Omega))$ for every $\delta\in[0,2)$.
Recall that the time continuity of the velocity field into $V_{div}$ is a consequence
of the fact that $u\in C_w([0,\infty);V_{div})$ and of the following differential identity
\begin{align}
&\frac{1}{2}\frac{d}{dt}\Vert\nabla u\Vert^2+\nu\Vert Su\Vert^2+(Bu,Su)=(\mu\nabla\varphi,Su)+(h,Su),
\label{est47}
\end{align}
which is deduced by testing equation \eqref{sy3} by $Su$.
}
\end{oss}

\begin{oss}
\label{eventbd}
{\upshape
If the condition $\varphi_0\in L^{\infty}(\Omega)$ in the first part of Theorem \ref{main} is removed, a boundedness estimate
for the order parameter $\varphi$ can still be recovered. In particular, it can be proved (see \cite[Lemma 2.10]{GG})
that for every $t_0>0$ there
exists a constant $\overline{C}_{m,t_0}>0$, where $m$ is such that $|\overline{\varphi}_0|\leq m$, such that
$$\sup_{t\geq 2t_0}\Vert\varphi(t)\Vert_{L^\infty(\Omega)}\leq \overline{C}_{m,t_0}.$$
Moreover, \eqref{reg1}--\eqref{reg4} still hold provided the time interval $(0,T)$ is replaced by $(2t_0,T)$, for every $T>2t_0$.
}
\end{oss}

\begin{oss}
{\upshape
In Theorem \ref{main} condition $J\in W^{2,1}(\mathbb{R}^2)$ is actually needed to ensure the regularity property
$\varphi\in L^\infty(0,T;H^2(\Omega))$ only.}
\end{oss}

\begin{proof}

We shall carry out the proof by providing some formal regularization estimates. The argument can be made rigorous by
means, e.g., of a Faedo-Galerkin approximation technique (see \cite{CFG} for details).

We first observe that the property $\varphi\in L^\infty(\Omega\times(0,T))$ can be obtained
by exploiting the same argument used in \cite[Theorem 2.1]{BH1}.
Indeed, by multiplying \eqref{sy1} by $\varphi|\varphi|^{p-1}$ and integrating on $\Omega$ the resulting equation,
the contribution of the convective term vanishes due to the incompressibility condition \eqref{sy4} and
the proof of \cite[Theorem 2.1]{BH1} entails
\begin{align}
\sup_{t\in(0,T)}\Vert\varphi(t)\Vert_{L^\infty(\Omega)}\leq \overline{C},
\label{phibd}
\end{align}
where the constant $\overline{C}$ depends on the initial conditions, in particular
on $\Vert u_0\Vert$, on $\Vert\varphi_0\Vert_{L^\infty(\Omega)}$ and on $T$
(see \cite[Estimate (2.28)]{BH1}).
Furthermore, if $h\in L^2_{tb}(0,\infty;G_{div})$ then, thanks to the dissipative estimate \eqref{dissest},
we have $\sup_{t\geq 0}\Vert\varphi(t)\Vert_{L^{2+2q}(\Omega)}\leq \overline{C}$, the constant $\overline{C}$ being dependent
on the initial data and on $h$ only. Hence, due to \cite[Estimate (2.28)]{BH1}, the constant $\overline{C}$ in \eqref{phibd}
does not depend on $T$.

As far as the regularity of the velocity $u$ is concerned, notice that, since the Korteweg-force term
$\mu\nabla\varphi\in L^2(0,T;L^2(\Omega)^2)$,
 then by applying \cite[Theorem 3.10]{T}, we immediately obtain \eqref{reg1}$_1$ and \eqref{reg2}$_2$.

Henceforth we shall denote by $c$ a positive constant which depends only on $J$, $F$ and $\Omega$, while $\overline{c}$
will denote a positive constant depending on $J$, $F$, $\Omega$ and also on the initial conditions $u_0$ and $\varphi_0$
(in particular on $\Vert\nabla u_0\Vert$ and on $\Vert\varphi_0\Vert_{L^\infty(\Omega)}$).
The values of both $c$ and $\overline{c}$ may possibly vary from line to line, even within the same estimate.
We shall divide the proof into three main steps.

\textbf{ Step 1. Estimate of $\varphi_t$ in $L^2(0,T;H)$}

We multiply \eqref{sy1} by $\mu_t$ in $H$ and get
\begin{align}
&\int_{\Omega}\varphi_t\mu_t+\int_{\Omega}(u\cdot\nabla\varphi)\mu_t+\frac{1}{2}\frac{d}{dt}\Vert\nabla\mu\Vert^2\nonumber\\
&=\int_{\Omega}(a+F''(\varphi))\varphi_t^2-(\varphi_t,J\ast\varphi_t)
+\int_{\Omega}(u\cdot\nabla\varphi)\mu_t+\frac{1}{2}\frac{d}{dt}\Vert\nabla\mu\Vert^2=0.
\label{diffid1}
\end{align}
Now, we have
\begin{align}
&\Big|\int_{\Omega}(u\cdot\nabla\varphi)\mu_t\Big|
=\Big|\int_{\Omega}(u\cdot\nabla\varphi)(a\varphi_t-J\ast\varphi_t+F''(\varphi)\varphi_t)\Big|\nonumber\\
&\leq\frac{c_0}{4}\Vert\varphi_t\Vert^2+\overline{c}\Vert u\Vert_{H^2}^2\Vert\nabla\varphi\Vert^2,
\label{est1}
\end{align}
and
\begin{align}
&|(\varphi_t,J\ast\varphi_t)|=|(-u\cdot\nabla\varphi+\Delta\mu,J\ast\varphi_t)|\nonumber\\
&\leq |(u\cdot\nabla\varphi,J\ast\varphi_t)|+|(\nabla\mu,\nabla J\ast\varphi_t)|\nonumber\\
&\leq \frac{c_0}{4}\Vert\varphi_t\Vert^2+c\Vert u\Vert_{H^2}^2\Vert\nabla\varphi\Vert^2+c\Vert\nabla\mu\Vert^2.
\label{est2}
\end{align}

Plugging \eqref{est1} and \eqref{est2} into \eqref{diffid1}, using assumption (H2) and integrating the resulting estimate in time between
$0$ and $t$, we obtain
\begin{align}
&\frac{1}{2}\Vert\nabla\mu\Vert^2+\frac{c_0}{2}\int_0^t\Vert\varphi_t\Vert^2 d\tau\leq\frac{1}{2}\Vert\nabla\mu_0\Vert^2
+\int_0^t \overline{c}\Vert u\Vert_{H^2}^2\Vert\nabla\varphi\Vert^2 d\tau+c\int_0^t\Vert\nabla\mu\Vert^2 d\tau,
\label{est3}
\end{align}
and on account of the following
\begin{align}
&\Vert\nabla\mu\Vert^2\geq\frac{c_0^2}{4}\Vert\nabla\varphi\Vert^2-c\Vert\varphi\Vert^2,
\label{est20}
\end{align}
from \eqref{est3} we are led to the differential inequality
\begin{align}
&\Vert\nabla\mu\Vert^2\leq \Vert\nabla\mu_0\Vert^2+\overline{c}_T+\int_0^t m(\tau)\Vert\nabla\mu(\tau)\Vert^2 d\tau,\qquad\forall t\in[0,T],
\end{align}
where $m:=c\left(\Vert u\Vert_{H^2}^2+1\right)\in L^1(0,T)$, for all $T>0$.
Thus the standard Gronwall lemma gives
\begin{align}
\nabla\mu\in L^{\infty}(0,T;H), \qquad\forall T>0,
\label{nablamu1}
\end{align}
so that, using also \eqref{est3}, we infer
\begin{align}
\varphi\in L^{\infty}(0,T;V), \qquad \varphi_t\in L^2(0,T;H),\qquad\forall T>0.
\label{est12}
\end{align}

This concludes the proof of \eqref{reg1} and \eqref{reg2}.

\textbf{Step 2. Estimate of $\varphi_t$ in $L^{\infty}(0,T;H)$}
$\\$
We differentiate \eqref{sy1} with respect to time and multiply the resulting identity in $H$ by $\mu_t$. This yields
\begin{align}
&\int_{\Omega}\varphi_{tt}\mu_t+\int_{\Omega}\mu_t u_t\cdot\nabla\varphi+\int_{\Omega}\mu_t u\cdot\nabla\varphi_t+\Vert\nabla\mu_t\Vert^2=0,
\label{est49}
\end{align}
and, due to \eqref{sy4}, we obtain
\begin{align}
&\int_{\Omega}\varphi_{tt}\mu_t+\Vert\nabla\mu_t\Vert^2=\int_{\Omega}\varphi_t u\cdot\nabla\mu_t+\int_{\Omega}\varphi u_t\cdot\nabla\mu_t,
\label{est50}
\end{align}
which entails
\begin{align}
&\int_{\Omega}\varphi_{tt}\mu_t+\frac{1}{2}\Vert\nabla\mu_t\Vert^2\leq\int_{\Omega}(\varphi_t^2 u^2+\varphi^2 u_t^2).
\label{est4}
\end{align}

Observe now that
\begin{align}
&\int_{\Omega}\varphi_{tt}\mu_t=\int_{\Omega}\varphi_{tt}(a\varphi_t-J\ast\varphi_t+F''(\varphi)\varphi_t)\nonumber\\
&=\frac{1}{2}\frac{d}{dt}\int_{\Omega} a\varphi_t^2
-(J\ast\varphi_t,-u_t\cdot\nabla\varphi-u\cdot\nabla\varphi_t+\Delta\mu_t)
+\int_{\Omega} F''(\varphi)\varphi_t\varphi_{tt}\nonumber\\
&=\frac{1}{2}\frac{d}{dt}\int_{\Omega} (a+F''(\varphi))\varphi_t^2
-(\nabla J\ast\varphi_t,u_t\varphi)-(\nabla J\ast\varphi_t,u\varphi_t)\nonumber\\
&+(\nabla J\ast\varphi_t,\nabla\mu_t)
-\frac{1}{2}\int_{\Omega}F'''(\varphi)\varphi_t^3.
\label{est51}
\end{align}
On the other hand we have
\begin{align}
&|(\nabla J\ast\varphi_t,u_t\varphi)|\leq\Vert\nabla J\Vert_{L^1}\Vert u_t\Vert\Vert\varphi\Vert_{L^{\infty}}\Vert\varphi_t\Vert
\leq\frac{1}{2}\Vert u_t\Vert^2\Vert\varphi_t\Vert^2+\overline{c},\nonumber\\
&|(\nabla J\ast\varphi_t,u\varphi_t)|\leq \Vert\nabla J\Vert_{L^1}\Vert u \Vert_{L^{\infty}}\Vert\varphi_t\Vert^2
\leq c\Vert u \Vert_{H^2}\Vert\varphi_t\Vert^2,\nonumber\\
&|(\nabla J\ast\varphi_t,\nabla\mu_t)|\leq\frac{1}{4}\Vert\nabla\mu_t\Vert^2+\Vert\nabla J\Vert_{L^1}^2\Vert\varphi_t\Vert^2.\nonumber
\end{align}
Therefore from \eqref{est4} we get
\begin{align}
\frac{1}{2}\frac{d}{dt}\int_{\Omega} (a+F''(\varphi))\varphi_t^2+\frac{1}{4}\Vert\nabla\mu_t\Vert^2
&\leq c(\Vert u \Vert_{H^2}^2+\Vert u \Vert_{H^2}+\Vert u_t\Vert^2+1)\Vert\varphi_t\Vert^2\nonumber\\
&+\Vert\varphi\Vert_{L^{\infty}}^2\Vert u_t\Vert^2+\frac{1}{2}\int_{\Omega}F'''(\varphi)\varphi_t^3+\overline{c}.
\label{est8}
\end{align}
The integral term containing $\varphi_t^3$ can be estimated by means of Gagliardo-Nirenberg inequality in dimension two, that is,
\begin{equation}
\Big|\frac{1}{2}\int_{\Omega}F'''(\varphi)\varphi_t^3\Big|\leq \overline{c}\Vert\varphi_t\Vert_{L^3}^3
\leq \overline{c}(\Vert\varphi_t\Vert^3+\Vert\varphi_t\Vert^2\Vert\nabla\varphi_t\Vert)
\leq\frac{c_0^2}{32}\Vert\nabla\varphi_t\Vert^2+\overline{c}\Vert\varphi_t\Vert^4+\overline{c}.
\label{est9}
\end{equation}
We now need to estimate $\nabla\varphi_t$ in terms of $\nabla\mu_t$. In order to do that, let us
first control $\nabla\varphi$ in terms of $\nabla\mu$ in $L^p$, for every $2\leq p<\infty$.
We then take the gradient of $\mu=a\varphi-J\ast\varphi+F'(\varphi)$, multiply it
by $\nabla\varphi|\nabla\varphi|^{p-2}$ and integrate the resulting identity on $\Omega$. We get
\begin{align}
&\int_{\Omega}\nabla\varphi|\nabla\varphi|^{p-2}\cdot\nabla\mu=\int_{\Omega}(a+F''(\varphi))|\nabla\varphi|^p
+\int_{\Omega}(\varphi\nabla a-\nabla J\ast\varphi)\cdot\nabla\varphi|\nabla\varphi|^{p-2},\nonumber
\end{align}
and so, by (H2), we find
\begin{align}
c_0\Vert\nabla\varphi\Vert_{L^p}^p
&\leq\Vert\nabla\varphi\Vert_{L^p}^{p-1}\Vert\nabla\mu\Vert_{L^p}
+(\Vert\nabla a\Vert_{L^{\infty}}+\Vert \nabla J\Vert_{L^1})\Vert\varphi\Vert_{L^p}\Vert\nabla\varphi\Vert_{L^p}^{p-1}\nonumber\\
&\leq\frac{c_0}{2}\Vert\nabla\varphi\Vert_{L^p}^p+c\Vert\nabla\mu\Vert_{L^p}^p+c(\Vert\nabla a\Vert_{L^{\infty}}+\Vert \nabla J\Vert_{L^1})^p\nonumber
\Vert\varphi\Vert_{L^p}^p.
\end{align}
We therefore obtain
\begin{align}
&\Vert\nabla\varphi\Vert_{L^p}\leq c\Vert\nabla\mu\Vert_{L^p}+\overline{c},
\label{est5}
\end{align}
with $\overline{c}$ depending also on $p$.
We now see that the $L^p-$norm of $\nabla\mu$ can be estimated in terms of the $L^2-$norm of $\varphi_t$.
Indeed, using once more the two dimensional Gagliardo-Nirenberg inequality, we infer
\begin{align}
&\Vert\nabla\mu\Vert_{L^p}\leq c\Vert\nabla\mu\Vert^{2/p}\Vert\nabla\mu\Vert_{H^1}^{1-2/p}\nonumber\\
&\leq c\Vert\nabla\mu\Vert^{2/p}\Vert\mu\Vert_{H^2}^{1-2/p}\leq c\Vert\nabla\mu\Vert^{2/p}(\Vert\Delta\mu\Vert^{1-2/p}+\Vert\mu\Vert^{1-2/p})\nonumber\\
&\leq \overline{c}(\Vert\varphi_t\Vert^{1-2/p}+\Vert u\cdot\nabla\varphi\Vert^{1-2/p}+1)\nonumber\\
&\leq \overline{c}(\Vert\varphi_t\Vert^{1-2/p}+\Vert u\Vert_{L^q}^{1-2/p}\Vert\nabla\varphi\Vert_{L^p}^{1-2/p}+1),\nonumber
\end{align}
where $p^{-1}+q^{-1}=1/2$ and where we have taken into account \eqref{nablamu1} and the fact that the $H^2-$norm of $\mu$ is equivalent to the $L^2-$ norm of $-\Delta\mu +\mu$,
due to \eqref{sy5}.
By \eqref{est5} we therefore deduce the desired estimate
\begin{align}
&\Vert\nabla\mu\Vert_{L^p}\leq \overline{c}(1+\Vert\varphi_t\Vert^{1-2/p}).
\label{est13}
\end{align}

We now take the gradient of $\mu_t$ and multiply it in $L^2$ by $\nabla\varphi_t$. We get
\begin{align}
&\int_{\Omega}\nabla\mu_t\cdot\nabla\varphi_t=\int_{\Omega}(a+F''(\varphi))|\nabla\varphi_t|^2\nonumber\\
&+\int_{\Omega}(\nabla a\varphi_t-\nabla J\ast\varphi_t)
\cdot\nabla\varphi_t+\int_{\Omega} F'''(\varphi)\varphi_t\nabla\varphi\cdot\nabla\varphi_t.
\label{est6}
\end{align}
Observe that we have
\begin{align}
&\Big|\int_{\Omega} F'''(\varphi)\varphi_t\nabla\varphi\cdot\nabla\varphi_t\Big|\leq c\Vert\varphi_t\Vert_{L^3}\Vert\nabla\varphi\Vert_{L^6}\Vert\nabla\varphi_t\Vert\nonumber\\
&\leq \overline{c}\Big(\Vert\varphi_t\Vert+\Vert\varphi_t\Vert^{2/3}\Vert\nabla\varphi_t\Vert^{1/3}\Big)   \Big(1+\Vert\varphi_t\Vert^{2/3}\Big)\Vert\nabla\varphi_t\Vert\nonumber\\
&\leq \overline{c}\Big(\Vert\varphi_t\Vert^{5/3}\Vert\nabla\varphi_t\Vert+\Vert\varphi_t\Vert^{4/3}\Vert\nabla\varphi_t\Vert^{4/3}
+\Vert\varphi_t\Vert^{2/3}\Vert\nabla\varphi_t\Vert^{4/3}+\Vert\varphi_t\Vert\Vert\nabla\varphi_t\Vert\Big)\nonumber\\
&\leq\frac{c_0}{4}\Vert\nabla\varphi_t\Vert^2+\overline{c}\Vert\varphi_t\Vert^4+\overline{c},
\label{est7}
\end{align}
Thus from \eqref{est6} and \eqref{est7}
and using also (H2), we deduce
\begin{align}
&\frac{1}{c_0}\Vert\nabla \mu_t\Vert^2+\frac{c_0}{4}\Vert\nabla\varphi_t\Vert^2\geq\Vert\nabla\mu_t\Vert\Vert\nabla\varphi_t\Vert
\geq c_0\Vert\nabla\varphi_t\Vert^2-\frac{c_0}{4}\Vert\nabla\varphi_t\Vert^2-\overline{c}\Vert\varphi_t\Vert^2\nonumber\\
&-\frac{c_0}{4}\Vert\nabla\varphi_t\Vert^2-\overline{c}\Vert\varphi_t\Vert^4-\overline{c},\nonumber
\end{align}
so that
\begin{align}
&\frac{4}{c_0^2}\Vert\nabla \mu_t\Vert^2\geq\Vert\nabla\varphi_t\Vert^2-\overline{c}\Vert\varphi_t\Vert^4-\overline{c}.
\label{est10}
\end{align}

We now go back to \eqref{est8}. By combining \eqref{est9} and \eqref{est10} we obtain
\begin{align}
&\frac{1}{2}\frac{d}{dt}\int_{\Omega} (a+F''(\varphi))\varphi_t^2+\frac{1}{8}\Vert\nabla\mu_t\Vert^2
\leq\alpha(t)\Vert\varphi_t\Vert^2+\overline{c}\Vert\varphi_t\Vert^4+\beta(t)+\overline{c},\label{est11}
\end{align}
where $\alpha:=c(\Vert u \Vert_{H^2}^2+\Vert u \Vert_{H^2}+\Vert u_t\Vert^2+1)$ and $\beta:=\Vert\varphi\Vert_{L^{\infty}}^2\Vert u_t\Vert^2$.
We have $\alpha$, $\beta\in L^1(0,T)$. From \eqref{est11} we can easily infer the desired estimate. Indeed, let
us multiply \eqref{est11} by $(1+\int_{\Omega}(a+F''(\varphi))\varphi_t^2)^{-1}$ and get
\begin{align}
&\frac{1}{2}\frac{d}{dt}\log\Big(1+\int_{\Omega}(a+F''(\varphi))\varphi_t^2\Big)\leq\frac{1}{c_0}\alpha(t)+\frac{\overline{c}\Big(\int_{\Omega}\varphi_t^2\Big)^2}
{1+\int_{\Omega}(a+F''(\varphi))\varphi_t^2}+\beta(t)+\overline{c}\nonumber\\
&\leq\frac{1}{c_0}\alpha(t)+\beta(t)+\overline{c}\Vert\varphi_t\Vert^2+\overline{c}.\nonumber
\end{align}
Integrating this last inequality between $0$ and $t\in(0,T)$ and using the second of \eqref{est12}
and the fact that $\varphi_t(0)\in H$ (since $\varphi_0\in H^2(\Omega)$) we therefore deduce that
\begin{align}
\varphi_t\in L^{\infty}(0,T;H),\qquad\forall T>0.
\label{est14}
\end{align}
In particular, on account of \eqref{est5} and \eqref{est13}, we also have
\begin{align}
&\nabla\mu,\nabla\varphi\in L^{\infty}(0,T;L^p(\Omega)),\qquad\forall T>0,\qquad 2\leq p<\infty.
\label{est15}
\end{align}
Furthermore, by integrating \eqref{est11} between 0 and $t\in[0,T]$ and using \eqref{est10} and \eqref{est14}, we also get
\begin{align}
&\varphi_t\in L^2(0,T;V).
\end{align}
By comparison in \eqref{sy1} we can finally obtain estimates for $\mu$ and $\varphi$ in $L^{\infty}(0,T;H^2(\Omega))$.
Indeed, we have
\begin{align}
&\Vert\Delta\mu\Vert\leq\Vert\varphi_t\Vert+c\Vert\nabla u\Vert\Vert\nabla\varphi\Vert_{L^p},
\label{est43}
\end{align}
which implies that $\Delta\mu\in L^{\infty}(0,T;L^2(\Omega))$, thanks to \eqref{est14} and \eqref{est15}.
Recalling \eqref{sy5} and the smoothness of $\partial\Omega$, we also have
\begin{align}
&\mu\in L^{\infty}(0,T;H^2(\Omega)).
\label{est16}
\end{align}

Apply now the second derivative operator $\partial^2_{ij}:=\frac{\partial^2}{\partial x_i\partial x_j}$ to \eqref{sy2},
multiply the resulting identity by $\partial^2_{ij}\varphi$ and integrate on $\Omega$. Using the assumption $J\in W^{2,1}(\mathbb{R}^2)$, we get
\begin{align}
&\int_{\Omega}\partial^2_{ij}\mu\partial^2_{ij}\varphi
=\int_{\Omega}(a+F''(\varphi))(\partial^2_{ij}\varphi)^2+\int_{\Omega}(\partial_i a\partial_j\varphi+\partial_j a\partial_i\varphi)\partial^2_{ij}\varphi
\nonumber\\
&+\int_{\Omega}(\varphi\partial^2_{ij}a -\partial^2_{ij}J\ast\varphi)\partial^2_{ij}\varphi
+\int_{\Omega}F'''(\varphi)\partial_i\varphi\partial_j\varphi\partial^2_{ij}\varphi.\nonumber
\end{align}
From this identity, by means of (H2) and \eqref{est15} it is easy to obtain
\begin{align}
&\Vert\partial^2_{ij}\mu\Vert^2\geq\frac{c_0^2}{4}\Vert\partial^2_{ij}\varphi\Vert^2-\overline{c}.
\label{est44}
\end{align}
Such estimate together with \eqref{est16} entail
\begin{align}
&\varphi\in L^{\infty}(0,T;H^2(\Omega)).
\end{align}

\textbf{Step 3. Continuous dependence and uniqueness of strong solutions}
$\\$
Let us consider two strong solutions $z_1:=[u_1,\varphi_1]$ and $z_2:=[u_2,\varphi_2]$ corresponding to initial data
 $z_{01}:=[u_{01},\varphi_{01}]$ and $z_{02}:=[u_{02},\varphi_{02}]$ and to external forces $h_1$ and $h_2$, respectively.
 Taking the difference between the variational formulation of \eqref{sy1} and \eqref{sy2} written for each solution and setting $u:=u_2-u_1$, $\varphi:=\varphi_2-\varphi_1$,
$\mu:=\mu_2-\mu_1$ and $h:=h_2-h_1$, we have
\begin{align}
&\langle u_t,v\rangle+\nu(\nabla u,\nabla v)+b(u_2,u_2,v)-b(u_1,u_1,v)=-(\varphi_2\nabla\mu_2,v)+(\varphi_1\nabla\mu_1,v)+(h,v)\\
&\langle\varphi_t,\psi\rangle+(\nabla\mu,\nabla\psi)=(u_2\varphi_2,\nabla\psi)-(u_1\varphi_1,\nabla\psi),
\end{align}
for every $v\in V_{div}$ and every $\psi\in V$. Let us choose $v=u$ and $\psi=\mathcal{N}\varphi$ and sum the first resulting identity
to the second one multiplied by $\gamma$, where the positive constant $\gamma$ will be suitably chosen. After some easy calculations we obtain
\begin{align}
&\frac{1}{2}\frac{d}{dt}\Vert u\Vert^2+\nu\Vert\nabla u\Vert^2+b(u_2,u_2,u)-b(u_1,u_1,u)
+\frac{\gamma}{2}\frac{d}{dt}\Vert\varphi\Vert_{V_0'}^2+\gamma(\varphi,\mu)\nonumber\\
&=-(\varphi\nabla\mu_2,u)-(\varphi_1\nabla\mu,u)+\gamma(u_2,\varphi\nabla\mathcal{N}\varphi)+\gamma (u,\varphi_1\nabla\mathcal{N}\varphi)
+(h,u).
\label{uniq1}
\end{align}
Notice that
\begin{align}
&\gamma(\varphi,\mu)=\gamma(\varphi,a\varphi-J\ast\varphi+F'(\varphi_2)-F'(\varphi_1))
\geq c_0\gamma\Vert\varphi\Vert^2-\gamma(\varphi,J\ast\varphi)\nonumber\\
&\geq c_0\gamma\Vert\varphi\Vert^2-\gamma\Vert\varphi\Vert_{V_0'}\Vert J\Vert_V\Vert\varphi\Vert
\geq c_0\gamma\Vert\varphi\Vert^2-\Vert\varphi\Vert^2-c\gamma^2\Vert\varphi\Vert_{V_0'}^2.
\label{est17}
\end{align}
Furthermore, as far as the first two terms on the right hand side of \eqref{uniq1} are concerned, we have
\begin{align}
&|(\varphi\nabla\mu_2,u)|\leq\Vert\varphi\Vert\Vert\nabla\mu_2\Vert_{L^4}\Vert u\Vert_{L^4}\leq\frac{\nu}{4}\Vert\nabla u\Vert^2
+c\Vert\nabla\mu_2\Vert_{L^4}^2\Vert\varphi\Vert^2,\\
&|(\varphi_1\nabla\mu,u)|=|(\mu\nabla\varphi_1,u)|\leq\Vert\mu\Vert\Vert\nabla\varphi_1\Vert_{L^4}\Vert u\Vert_{L^4}
\leq\frac{\nu}{4}\Vert\nabla u\Vert^2
+c\Vert\nabla\varphi_1\Vert_{L^4}^2\Vert\varphi\Vert^2,
\end{align}
where we have used the bound
$$\Vert\mu\Vert=\Vert a\varphi-J\ast\varphi+F'(\varphi_2)-F'(\varphi_1)\Vert\leq 2\Vert a\Vert_{L^{\infty}}\Vert\varphi\Vert+c\Vert\varphi\Vert\leq c\Vert\varphi\Vert.$$
The last two terms on the right hand side of \eqref{uniq1} can be estimated as follows
\begin{align}
&|\gamma(u_2,\varphi\nabla\mathcal{N}\varphi)|\leq\gamma\Vert u_2\Vert_{L^{\infty}}\Vert\varphi\Vert\Vert\nabla\mathcal{N}\varphi\Vert
\leq c\gamma\Vert u_2\Vert_{H^2}\Vert\varphi\Vert\Vert\varphi\Vert_{V_0'}\nonumber\\
&\leq\Vert\varphi\Vert^2+c\gamma^2\Vert u_2\Vert_{H^2}^2\Vert\varphi\Vert_{V_0'}^2,\\
&|\gamma (u,\varphi_1\nabla\mathcal{N}\varphi)|\leq \frac{\gamma}{2}\Vert u\Vert^2+\frac{\gamma}{2}\Vert\varphi_1\Vert_{L^{\infty}}^2\Vert\varphi\Vert_{V_0'}^2.
\end{align}
Consider the trilinear forms on the left hand side of \eqref{uniq1}. By \eqref{ftril2} we have
\begin{align}
&b(u_2,u_2,u)-b(u_1,u_1,u)=b(u,u_1,u)\leq c\Vert u\Vert\Vert \nabla u_1\Vert\Vert\nabla u\Vert\nonumber\\
&\leq\frac{\nu}{4}\Vert\nabla u\Vert^2+c\Vert\nabla u_1\Vert^2\Vert u\Vert^2
\label{est18}
\end{align}
Plugging \eqref{est17}--\eqref{est18} into \eqref{uniq1} we get
\begin{align}
&
\frac{1}{2}\frac{d}{dt}\Big(\Vert u\Vert^2+\gamma\Vert\varphi\Vert_{V_0'}^2\Big)+\frac{\nu}{8}\Vert\nabla u\Vert^2
+\gamma c_0\Vert\varphi\Vert^2\leq c(1+\Vert\nabla\varphi_1\Vert_{L^4}^2+\Vert\nabla\mu_2\Vert_{L^4}^2)\Vert\varphi\Vert^2\nonumber\\
&
+c\gamma(\gamma\Vert u_2\Vert_{H^2}^2+\Vert\varphi_1\Vert_{L^{\infty}}^2+\gamma)\Vert\varphi\Vert_{V_0'}^2
+c(\gamma+\Vert\nabla u_1\Vert^2)\Vert u\Vert^2+\frac{2}{\nu\lambda_1}\Vert h\Vert^2.
\label{est19}
\end{align}
Thanks to \eqref{est15}, we can now choose $\gamma=\gamma_{\ast}$ such that
$$\Gamma_{\ast}:=c_0\gamma_{\ast}-c(1+\Vert\nabla\varphi_1\Vert_{L^{\infty}(0,T;L^4(\Omega))}^2+\Vert\nabla\mu_2\Vert_{L^{\infty}(0,T;L^4(\Omega))}^2)>0.$$
Hence from \eqref{est19} we deduce
\begin{align}
\label{th}
&
\frac{1}{2}\frac{d}{dt}\Big(\Vert u\Vert^2+\gamma_{\ast}\Vert\varphi\Vert_{V_0'}^2\Big)+\frac{\nu}{8}\Vert\nabla u\Vert^2
+\Gamma_{\ast}\Vert\varphi\Vert^2\leq\eta(t)\Big(\Vert u\Vert^2+\gamma_{\ast}\Vert\varphi\Vert_{V_0'}^2\Big)+\frac{2}{\nu\lambda_1}\Vert h\Vert^2,
\end{align}
where
$$\eta:=c(\Vert\nabla u_1\Vert^2+\gamma_{\ast}\Vert u_2\Vert_{H^2}^2+\Vert\varphi_1\Vert_{L^{\infty}}^2+\gamma_{\ast})\in L^1(0,T),\qquad\forall T>0.$$

The standard Gronwall lemma then yields
\begin{align}
&\Vert u(t)\Vert^2+\gamma_{\ast}\Vert\varphi(t)\Vert_{V_0'}^2\leq e^{2\int_0^t\eta(s)ds}\Big(\Vert u_0\Vert^2+\gamma_{\ast}\Vert\varphi_0\Vert_{V_0'}^2+
\frac{4}{\nu\lambda_1}\Vert h\Vert_{L^2(0,t;G_{div})}^2\Big),
\label{est28}
\end{align}
for every $t\in[0,T]$, where we have set $u_0:=u_{02}-u_{01}$ and $\varphi_0:=\varphi_{02}-\varphi_{01}$.
By integrating \eqref{th} between 0 and $t$ and taking \eqref{est28} into account, we also get
\begin{align}
&\frac{\nu}{4}\int_0^t\Vert\nabla u\Vert^2 d\tau+2\Gamma_\ast\int_0^t\Vert\varphi\Vert^2 d\tau\nonumber\\
&\leq\Big(\Vert u_0\Vert^2+\gamma_{\ast}\Vert\varphi_0\Vert_{V_0'}^2+
\frac{4}{\nu\lambda_1}\Vert h\Vert_{L^2(0,t;G_{div})}^2\Big)\Big(1+2e^{2\int_0^t\eta(s)ds}\int_0^t\eta(s)ds \Big),
\label{est29}
\end{align}
for every $t\in[0,T]$.
Finally, by combining \eqref{est28} and \eqref{est29}, we obtain \eqref{contdipest}.
\end{proof}

\begin{oss}\label{timereg2}
{\upshape
It is not difficult to see that the $\varphi$ component
of the strong solution to system \eqref{sy1}--\eqref{sy5}
satisfies
\begin{align}
&\varphi\in C([0,\infty);H^2(\Omega)).\label{phitimereg}
\end{align}
Indeed, by combining \eqref{est49}--\eqref{est51} and taking into account the regularity properties of the strong solution, we can see that $\int_\Omega\big(a+F''(\varphi)\big)\varphi_t^2$ is absolutely continuous
on $[0,\infty)$. Using (H2) and the fact that
$\varphi\in C([0,\infty);C(\overline{\Omega}))$ (see Remark \ref{timereg})
 we get $\Vert\varphi_t\Vert^2\in C([0,\infty))$.
 Now, \eqref{est16} and $\mu_t\in L^2_{loc}([0,\infty);V)$ imply that $\mu\in C([0,\infty);V)$
 and, by using \eqref{est16} again, we also have $\mu\in C_w([0,\infty);H^2(\Omega))$ so that
 $\Delta\mu\in C_w([0,\infty);H)$. Moreover,
 since $u\in C([0,\infty);L^4(\Omega))$ and $\nabla\varphi\in C_w([0,\infty);L^4(\Omega))$
 (cf. Remark \ref{timereg}), then we have $u\cdot\nabla\varphi\in C_w([0,\infty);H)$.
 Thus from \eqref{sy1} we deduce that $\varphi_t\in C_w([0,\infty);H)$
 and, on account of the continuity of $t\mapsto \Vert\varphi_t(t)\Vert$, then $\varphi_t\in C([0,\infty);H)$.
 Recall now that $\nabla\varphi\in C([0,\infty);H^\epsilon(\Omega)^2)$, for every $\epsilon\in [0,1)$
 (cf. Remark \ref{timereg}). Then, choosing $\epsilon\in[1/2,1)$, we have $\nabla\varphi\in C([0,\infty);L^4(\Omega)^2)$. Thus $u\cdot\nabla\varphi\in C([0,\infty);H)$ and so
 \eqref{sy1} yields $\Delta\mu\in C([0,\infty);H)$ which entails $\mu\in C([0,\infty);H^2(\Omega))$.
 This and the assumption $J\in W^{2,1}(\mathbb{R}^2)$ allow us to deduce
 \eqref{phitimereg}.
}
\end{oss}


\section{Uniform estimates and the global attractor}\setcounter{equation}{0}
In this section we establish some uniform in time regularization estimates by exploiting the results proved in the previous section.
As a consequence we deduce a regularity property for the global attractor
of the dynamical system generated by \eqref{sy1}--\eqref{sy5} whose existence
has been shown in \cite{FG1}.

\begin{prop}
Let $h\in L^2_{tb}(0,\infty;G_{div})$, $u_0\in V_{div}$, $\varphi_0\in V\cap L^{\infty}(\Omega)$
and suppose that (H1)-(H4) are satisfied.
Then, the weak solution $[u,\varphi]$ of Theorem \ref{main} satisfies
\begin{align}
& u\in L^{\infty}(0,\infty;V_{div})\cap L^2_{tb}(0,\infty;H^2(\Omega)^2),\quad\varphi \in L^\infty(\Omega\times(0,\infty))\cap L^\infty(0,\infty;V),
\label{reg5}\\
&u_t\in L^2_{tb}(0,\infty;G_{div}),\quad\varphi_t\in L^2_{tb}(0,\infty;H).\label{reg6}
\end{align}
Furthermore, suppose in addition that $F\in C^3(\mathbb{R})$ and that $\varphi_0\in H^2(\Omega)$.
Then, the unique strong solution of Theorem \ref{main} satisfies \eqref{reg5}, \eqref{reg6}
and, in addition,
\begin{align}
&\varphi\in L^{\infty}(0,\infty;W^{1,p}(\Omega)),\qquad 2\leq p<\infty,\label{reg7}\\
& \varphi_t\in L^{\infty}(0,\infty;H)\cap L^2_{tb}(0,\infty;V).\label{reg8}
\end{align}

If $J\in W^{2,1}(\mathbb{R}^2)$, we also have
\begin{align}
&\varphi\in L^{\infty}(0,\infty;H^2(\Omega)).
\label{reg9}
\end{align}
Moreover, there exists a constant $\Lambda_1=\Lambda_1(m)$, depending on $m$ (and on $F$, $J$, $\Omega$, $\nu$),
such that, for every initial data $z_0:=[u_0,\varphi_0]\in V_{div}\times H^2(\Omega)$, with $|\overline{\varphi}_0|\leq m$,
there exists a time $t^\ast:=t^\ast(\mathcal{E}(z_0))\geq 0$ such that the strong solution corresponding to $z_0$
satisfies
\begin{align}
&\Vert\nabla u(t)\Vert+\Vert\varphi(t)\Vert_{H^2(\Omega)}+\int_t^{t+1}\Vert u(s)\Vert_{H^2(\Omega)^2}
\leq\Lambda_1(m),\qquad\forall t\geq t^\ast.\label{secdissest}
\end{align}

\end{prop}

\begin{proof}

Let us first notice that, setting $z(t):=[u(t),\varphi(t)]$ and $z_0:=[u_0,\varphi_0]$, by integrating the energy identity \eqref{eniden} between $t$ and $t+1$ we have
\begin{align}
&\mathcal{E}(z(t+1))+\int_t^{t+1}\Big(\frac{\nu}{2}\Vert\nabla u\Vert^2+\Vert\nabla\mu\Vert^2\Big)d\tau\leq
\mathcal{E}(z(t))+\frac{1}{2\nu\lambda_1}\int_t^{t+1}\Vert h\Vert^2d\tau.
\end{align}
Therefore, using also the dissipative estimate \eqref{dissest}, we get
\begin{align}
\label{est21}
&\int_t^{t+1}\Big(\frac{\nu}{2}\Vert\nabla u\Vert^2+\Vert\nabla\mu\Vert^2\Big)d\tau\leq
\mathcal{E}(z_0)e^{-kt}+ F(m)|\Omega|+ K
\end{align}
where the constant $K$ depends on $\|h\|_{L^2_{tb}(0,\infty;G_{div})}$ and on $F$, $J$, $\Omega$, $\nu$.
Notice that the initial energy $\mathcal{E}(z_0)$ can be estimated as
$$\mathcal{E}(z_0)\leq\frac{1}{2}\Vert u_0\Vert^2+M\Vert\varphi_0\Vert^2+\int_{\Omega}F(\varphi_0),\qquad M:=\sup_{x\in\Omega}\int_{\Omega}|J(x-y)|dy.$$
From \eqref{est21}, setting $\Lambda_0(m):=F(m)|\Omega|+K+1$, we deduce that there exists a time $t_0=t_0(\mathcal{E}(z_0))>0$, given e.g.
by $t_0=\frac{1}{k}\log(\mathcal{E}(z_0)+c)$, where $\mathcal{E}(z_0)+c > 1$, such that
\begin{align}
\label{est31}
&\int_t^{t+1}\Big(\frac{\nu}{2}\Vert\nabla u\Vert^2+\Vert\nabla\mu\Vert^2\Big)d\tau\leq
\Lambda_0(m),\qquad\forall t\geq t_0.
\end{align}

We now establish the uniform in time version of estimates \eqref{reg1}$_1$ and \eqref{reg2}$_1$ for the velocity field.
To this aim,
notice first that \eqref{ftril3} implies (see also \cite[Lemma 3.8]{T})
$$\Vert Bu\Vert\leq c\Vert u\Vert^{1/2}\Vert\nabla u\Vert \Vert Su\Vert^{1/2},
 \qquad\forall u\in D(S)=H^2(\Omega)^2\cap V_{div}.$$
Therefore, by splitting the term $(Bu,Su)$ on the left hand side of \eqref{est47}
and using the estimate above, we get the following differential inequality
\begin{align}
&\frac{d}{dt}\Vert\nabla u\Vert^2+\nu\Vert Su\Vert^2\leq\frac{3}{\nu}\Vert\mu\nabla\varphi\Vert^2+\frac{3}{\nu}\Vert h\Vert^2+\sigma\Vert\nabla u\Vert^2,
\label{NSregin}
\end{align}
where $\sigma(t):=c_{\nu}\Vert u\Vert^2\Vert\nabla u\Vert^2$. 
Now, recalling Remark \ref{eventbd} (see also the proof of \cite[Lemma 2.10]{GG}),
the assumption $h\in L^2_{tb}(0,\infty;G_{div})$ and the dissipative estimate \eqref{dissest},
we know that there exists a constant $C_0(m)>0$ depending on $m$, and a time $t_1=t_1(\mathcal{E}(z_0))$ depending on $\mathcal{E}(z_0)$ such that
\begin{align}
&\sup_{t\geq t_1}\Vert\varphi(t)\Vert_{L^\infty(\Omega)}\leq C_0(m).
\label{est35}
\end{align}
Therefore we have
$\sup_{t\geq t_1}\Vert\mu(t)\Vert_{L^{\infty}(\Omega)}\leq C_1(m)$.
Then, using also \eqref{est31} and \eqref{est20}, we get
\begin{align}
&\int_t^{t+1}\Vert\mu(\tau)\nabla\varphi(\tau)\Vert^2 d\tau\leq C_2(m),\qquad\int_t^{t+1}\sigma(\tau)d\tau\leq C_3(m),
\label{est22}
\end{align}
for all $t\geq t_2:=\max\{t_0,t_1\}$.
Therefore, \eqref{est21} and \eqref{est22} allow us to apply
Lemma \ref{unifGronw} to the differential inequality \eqref{NSregin}
and we deduce that
\begin{align}
&\Vert\nabla u(t)\Vert^2\leq C_4(m):=\frac{2}{\nu}\Big(2C_2(m)+2\Vert h\Vert_{L^2_{tb}(0,\infty;G_{div})}^2+\Lambda_0(m)\Big)e^{C_3(m)},
\label{est32}
\end{align}
for all $t\geq t_3:=t_2+1$. Furthermore, by integrating \eqref{NSregin} between $t$ and $t+1$, for $t\geq t_3$, we obtain
\begin{align}
&c\nu\int_t^{t+1}\Vert u(s)\Vert_{H^2(\Omega)}^2 ds\leq
C_5(m):=(1+C_3(m))C_4(m)+\frac{4}{\nu}\Big(C_2(m)+\Vert h\Vert_{L^2_{tb}(0,\infty;G_{div})}^2\Big),
\label{est33}
\end{align}
for all $t\geq t_3$, where we have also used \cite[Lemma 3.7]{T}.
Estimates \eqref{est32} and \eqref{est33} in particular imply \eqref{reg5}$_1$.

Now, let us write \eqref{sy3} in the form $u_t=-Bu-\nu Su+\mu\nabla\varphi+h$ and observe that,
owing to \cite[Lemma 3.8]{T} (or \eqref{ftril3}), we have
\begin{align}
&\int_t^{t+1}\Vert Bu(s)\Vert^4 ds\leq\int_t^{t+1}\Vert u(s)\Vert^2\Vert\nabla u(s)\Vert^4\Vert Su(s) \Vert^2 ds
\leq C_6(m):=\frac{c}{\nu\lambda_1} C_4^3(m) C_5(m),\nonumber
\end{align}
for all $t\geq t_3$, and hence
\begin{align}
&\int_t^{t+1}\Vert u_t(s)\Vert^2 ds\leq C_7(m):=c\Big(C_6^{1/2}(m)+\nu C_5(m)+C_2(m)+\Vert h\Vert_{L^2_{tb}(0,\infty;G_{div})}^2\Big),
\label{est34}
\end{align}
for all $t\geq t_3$. Note that \eqref{est34} entails \eqref{reg6}$_1$.

We are now in a position to get uniform in time regularization estimates for $\varphi_t$ first in $L^2_{tb}(\tau,\infty;H)$ and then in $L^{\infty}(\tau,\infty;H)$, for some $\tau>0$.

Let us note first that, by combining \eqref{diffid1}--\eqref{est2} and taking \eqref{est35} into account, we obtain the following differential
inequality, for all $t>t_1$,
\begin{align}
&\frac{d}{dt}\Vert\nabla\mu\Vert^2+c_0\Vert\varphi_t\Vert^2\leq(C_8(m)\Vert u\Vert_{H^2}^2+c)\Vert\nabla\mu\Vert^2
+C_9(m)\Vert u\Vert_{H^2}^2\Vert\varphi\Vert^2.
\label{est38}
\end{align}
Observe that (cf. \eqref{est33})
\begin{align}
&\int_t^{t+1}(C_8(m)\Vert u(s)\Vert_{H^2}^2+c)ds\leq C_{10}(m):=\frac{1}{c\nu}C_5(m)C_8(m)+c,\label{est36}\\
&\int_t^{t+1}C_9(m)\Vert u(s)\Vert_{H^2}^2\Vert\varphi(s)\Vert^2 ds\leq C_{11}(m):=\frac{|\Omega|}{c\nu}C_0^2(m)C_5(m)C_9(m)\label{est37},
\end{align}
for all $t\geq t_3$. Then, using \eqref{est31} and \eqref{est36}, \eqref{est37}, we can apply the uniform Gronwall lemma to
\eqref{est38} in $[t_3,\infty)$ and get
\begin{align}
&\Vert\nabla\mu(t)\Vert^2\leq C_{12}(m):=(C_{11}(m)+\Lambda_0(m))e^{C_{10}(m)},\qquad\forall t\geq t_4:=t_3+1.
\label{est39}
\end{align}
Now, by integrating \eqref{est38} between $t$ and $t+1$, for $t\geq t_4$, we also deduce
\begin{align}
&c_0\int_t^{t+1}\Vert\varphi_t(s)\Vert^2 ds\leq C_{13}(m):=(1+C_{10}(m))C_{12}(m)+C_{11}(m),\qquad\forall t\geq t_4.
\label{est40}
\end{align}
Estimates \eqref{est39} and \eqref{est40} imply, in particular, \eqref{reg5}$_2$ and \eqref{reg6}$_2$, respectively.

Let us now consider estimate \eqref{est11}. Set
$$\Phi(t):=\frac{1}{2}\int_{\Omega}(a+F''(\varphi(t)))\varphi_t^2(t),$$
and notice that, on account of \eqref{est35}, we have
\begin{align}
&\frac{c_0}{2}\Vert\varphi_t(t)\Vert^2\leq\Phi(t)\leq C_{14}(m)\Vert\varphi_t(t)\Vert^2,\qquad\forall t\geq t_1.
\label{est26}
\end{align}
Then, by arguing as in the previous section and taking \eqref{est35} into account, we easily see that \eqref{est11} can be rewritten as follows
\begin{align}
&\frac{d}{dt}\Phi(t)+\frac{1}{8}\Vert\nabla\mu_t\Vert^2\leq\omega(t)\Phi(t)+\beta(t)+C_{15}(m),\qquad\forall t\geq t_1,
\label{est27}
\end{align}
where $\omega(t):=\alpha(t)+C_{16}(m)\Phi(t)$, and $\alpha,\beta$ the same as in \eqref{est11}.
Then, by using \eqref{est26}, \eqref{est40}, \eqref{est33} and \eqref{est34}, we have
\begin{align}
&\int_t^{t+1}\Phi(s)ds\leq C_{17}(m):=\frac{1}{c_0}C_{13}(m)C_{14}(m),\\
&\int_t^{t+1}\omega(s)ds\leq C_{18}(m):=c\Big(\frac{1}{\nu}C_5(m)+C_7(m)+C_{16}(m)C_{17}(m)+1\Big),\\
&\int_t^{t+1}\beta(s)ds\leq C_{19}(m):=C_0^2(m)C_7(m),
\end{align}
for all $t\geq t_4$.
By applying once more the uniform Gronwall lemma to \eqref{est27} in the interval $[t_4,\infty)$, we deduce
\begin{align}
&\Vert\varphi_t(t)\Vert^2\leq C_{20}(m):=\frac{2}{c_0}\Big(C_{15}(m)+C_{17}(m)+C_{19}(m)\Big)e^{C_{18}(m)},
\label{est41}
\end{align}
for all $t\geq t_5:=t_4+1$.
Then, by integrating \eqref{est27} between $t$ and $t+1$, for $t\geq t_5$,
and using \eqref{est26}, \eqref{est41} and \eqref{est10} (written with a constant $C_{21}(m)$ in place of $\overline{c}$, for $t\geq t_1$,
due to \eqref{est35}), we also find
\begin{align}
&\int_t^{t+1}\Vert\nabla\varphi_t(s)\Vert^2 ds\leq C_{22}(m):=\frac{32}{c_0}\Big(C_{14}C_{20}C_{18}+C_{19}+C_{15}\Big)+\big(1+C_{20}^2\big)C_{21},
\label{est42}
\end{align}
for all $t\geq t_5$, where all $C_i$ depend on $m$. Observe that
estimates \eqref{est41} and \eqref{est42} yield \eqref{reg8}.

Furthermore, owing to \eqref{est5} and \eqref{est13},
we also have
\begin{align}
&\Vert\nabla\varphi(t)\Vert_{L^p(\Omega)^2}\leq C_{23}(m),\qquad\forall t\geq t_5,\qquad 2\leq p<\infty
\label{est45}
\end{align}
Finally, on account of \eqref{est43}, \eqref{est41} and \eqref{est32}, we obtain
\begin{align}
&\Vert\mu(t)\Vert_{H^2}\leq c\Vert-\Delta\mu(t)+\mu(t)\Vert\leq C_{24}(m):=c\big(C_1(m)+C_{20}^{1/2}(m)+C_4^{1/2}(m)C_{23}(m)\big),
\end{align}
for all $t\geq t_5$,
and recalling \eqref{est44}, provided that $J\in W^{2,1}(\mathbb{R}^2)$, we get
\begin{align}
&\Vert\varphi(t)\Vert_{H^2}\leq C_{25}(m),\qquad\forall t\geq t_5.\label{est46}
\end{align}
Estimates \eqref{est45} and \eqref{est46} yield \eqref{reg7}.
\end{proof}

Let us now recall the main result about the existence of the global attractor for weak solutions to system \eqref{sy1}--\eqref{sy5}
in the autonomous case (cf. \cite{FG1}).
Since the weak solutions to system \eqref{sy1}--\eqref{sy5} are not known to be unique but the energy identity holds,
the existence of the global attractor is achieved by using J.M. Ball's approach based on the notion of generalized semiflows
(cf. \cite{Ba}, to which we refer for the main definitions and results).

We assume that $h$ is time independent, i.e., $h\in G_{div}$, and, for $m\geq 0$ fixed, we introduce the metric space
\begin{align}
&\mathcal{X}_m:=G_{div}\times\mathcal{Y}_m,
\end{align}
where
\begin{equation}
\mathcal{Y}_m:=\{\varphi\in H\,:\, F(\varphi)\in L^1(\Omega),
\; |(\varphi,1)|\leq m\},
\end{equation}
The space $\mathcal{X}_m$ is endowed with the metric
\begin{equation*}
\mathbf{d}(z_2,z_1)=\|u_2-u_1\|+\|\varphi_2-\varphi_1\|
+\Big|\int_{\Omega}F(\varphi_2)-\int_{\Omega}F(\varphi_1)\Big|^{1/2},\quad\forall z_1,z_2\in\mathcal{X}_m,
\end{equation*}
where $z_1:=[u_1,\varphi_1]$ and $z_2:=[u_2,\varphi_2]$.

Suppose that (H1)--(H4) are satisfied and that $h\in G_{div}$. Let $\mathcal{G}_m$ be the set of all weak solutions
to system \eqref{sy1}--\eqref{sy6} from $[0,\infty)$ to $\mathcal{X}_m$ given by Theorem \ref{thm}
and corresponding to all initial data $z_0\in\mathcal{X}_m$.
Then, in \cite[Prop. 3 and Thm. 3]{FG1} it is proved that $\mathcal{G}_m$ is a generalized semiflow on $\mathcal{X}_m$ (i.e.,
$\mathcal{G}_m$ satisfies conditions (H1)--(H4) from \cite{Ba} in the space $\mathcal{X}_m$) which possesses a (unique)
global attractor $\mathcal{A}_m$.

Take $z_0\in\mathcal{X}_m$ and consider a weak solution $z:=[u,\varphi]\in C([0,\infty);\mathcal{X}_m)$ corresponding to $z_0$.
From \eqref{ei}, written with $t=\tau$, we know that for every $\tau>0$ there exists $t_\tau\in(0,\tau]$
such that $z(t_\tau)\in V_{div}\times V$. Thanks to Remark \ref{eventbd}, we can also assume that $\varphi(t_\tau)\in L^\infty(\Omega)$.
We can therefore write the differential inequality \eqref{est38} in $[t_\tau,\infty)$ and, by integrating \eqref{est38} between $t_\tau$ and $t>t_\tau$,
we can see that there exists $s_\tau\in(t_\tau,t]$ such that $\varphi_t(s_\tau)\in H$ and hence $\varphi(s_\tau)\in H^2(\Omega)$.
Summing up, introducing the (complete) metric space
\begin{align}
&\mathcal{X}_m^1:=V_{div}\times\mathcal{Y}_m^1,\qquad\mathcal{Y}_m^1:=\{\varphi\in H^2(\Omega):\:\:|(\varphi,1)|\leq m\},
\end{align}
endowed with the metric
\begin{equation*}
\mathbf{d_1}(z_2,z_1)=\|\nabla u_2-\nabla u_1\|+\|\varphi_2-\varphi_1\|_{H^2(\Omega)},
\quad\forall z_1,z_2\in\mathcal{X}_m^1,
\end{equation*}
then, for every $\tau>0$, there exists $s_\tau\in(0,\tau]$ such that $z(s_\tau)\in\mathcal{X}_m^1$
and starting from the time $s_\tau$ the weak solution corresponding to $z_0$ becomes a (unique) strong solution $z\in C([s_\tau,\infty);\mathcal{X}_m^1)$ (cf. Remarks \ref{timereg} and \ref{timereg2}).
Such a solution satisfies the dissipative estimate \eqref{secdissest} in $[s_\tau,\infty)$.
Let us consider a bounded in $\mathcal{X}_m$ subset $B\subset\mathcal{X}_m$. Choosing $\tau=1$ for every $z_0\in B$, then
every weak solution $z$ starting from $z_0\in B$ becomes (at a certain time $s_1\in(0,1]$ depending on $z_0$ and on the weak
solution considered from $z_0$) a strong solution satisfying \eqref{secdissest} in $[1,\infty)$.
We therefore deduce that there exists a time $t^\ast=t^\ast(B)\geq 1$ such that
\begin{align}
& z(t)\in\mathcal{B}_1(\Lambda_1(m)), \qquad\forall t\geq t^\ast,
\label{est48}
\end{align}
where $\mathcal{B}_1(\Lambda_1(m))$ is the closed ball in $\mathcal{X}_m^1$ given by
$$\mathcal{B}_1(\Lambda_1(m)):=\{w\in \mathcal{X}_m^1:\:\:\mathbf{d_1}(w,0)\leq\Lambda_1(m)\}.$$
This fact immediately implies that $\mathcal{A}_m\subset\mathcal{B}_1$. Indeed, we have
$\mbox{dist}_{\mathcal{X}_m^1}(T(t)\mathcal{A}_m,\mathcal{B}_1)=\mbox{dist}_{\mathcal{X}_m^1}(\mathcal{A}_m,\mathcal{B}_1)=0$,
which implies $\mathcal{A}_m\subset\overline{\mathcal{B}_1}^{\mathcal{X}_m^1}=\mathcal{B}_1$.
We recall that the multivalued evolution map $T(t)$ is defined, for every $t\geq 0$
and every subset $E\subset\mathcal{X}_m$, as (cf. \cite{Ba})
\begin{align}
&T(t)E:=\{z(t):\:\: z\in\mathcal{G}_m,\:\: z(0)\in E\}.
\end{align}
Summing up we have just proven the following regularity result for the global attractor
\begin{thm}
Let (H1)--(H4) be satisfied and assume that $h\in G_{div}$ is independent of time.
Then the global attractor $\mathcal{A}_m$ of the generalized semiflow $\mathcal{G}_m$
associated with system \eqref{sy1}--\eqref{sy5} is such that
$$ \mathcal{A}_m\subset\mathcal{B}_1(\Lambda_1(m)).$$
\end{thm}
Thus the global attractor is the union of all the
bounded complete trajectories which are strong solutions to \eqref{sy1}-\eqref{sy6}.

\section{Convergence to equilibria}\setcounter{equation}{0}
 In this section we shall prove that every weak solution to system \eqref{sy1}--\eqref{sy6}
 converges to a stationary solution as $t\to\infty$, provided that
 $F$ is real analytic and $h\equiv 0$.

Let us first consider the set of all stationary
solutions $z_\infty$ to system \eqref{sy1}--\eqref{sy5}, namely the set of
pairs
$z_\infty:=[0,\varphi_\infty]\in\mathcal{X}_m$ (for some $m\geq
0$), where $\varphi_\infty$ solves the integral equation
\begin{align}
&a\varphi_\infty-J\ast\varphi_\infty+F'(\varphi_\infty)=\mu_\infty,
\label{co13}
\end{align}
with some constant $\mu_\infty\in\mathbb{R}$ given necessarily by
$\mu_\infty=\overline{F'(\varphi_\infty)}$. Therefore we introduce
\begin{eqnarray}
\mathcal{E}_m
&=&\Big\{z_\infty=[0,\varphi_\infty]:\quad\varphi_\infty\in
H,\quad F(\varphi_\infty)\in
L^1(\Omega),\quad |\overline{\varphi_\infty}|\leq m,\nonumber\\
&
&\quad a\varphi_\infty-J\ast\varphi_\infty+F'(\varphi_\infty)-\overline{F'(\varphi_\infty)}=0\quad\mbox{a.e.
in }\Omega\Big\}.
\end{eqnarray}
We point out that, by using an easy iteration argument
from \eqref{co13}, on account that $F'$ has polynomial growth,
we can deduce that $\varphi_\infty\in L^\infty(\Omega)$.
The structure of the stationary set is rather complicated.
In particular, such a set may be a continuum (see \cite{FIP} for an example
and \cite{H} where the author proves the existence of solutions $\varphi_\infty$ to \eqref{co3} with $\overline{\varphi}_\infty=0$ in cylindrical bounded domains).
It is also worth observing that to every stationary solution
$z_\infty=[0,\varphi_\infty]$ there corresponds a stationary
pressure $\pi_\infty$ given by
$\pi_\infty=\overline{F'(\varphi_\infty)}\varphi_\infty+c$,
where $c\in\mathbb{R}$ is an arbitrary constant (cf. \eqref{sy3}).

We begin with the following preliminary but crucial result.

\begin{lem}
\label{omegalim}
Assume that (H1)--(H4) are satisfied. Take $z_0\in\mathcal{X}_m$
and let $z\in C([0,\infty);\mathcal{X}_m)$ be a weak
solution corresponding to $z_0$. Then, we have
\begin{align}
&\emptyset \neq\omega(z)\subset\mathcal{E}_m \label{co00}
\end{align}
and
\begin{align}
& u(t)\to 0\qquad\mbox{in }\:\:G_{div},\qquad\mbox{as
}\:\:t\to\infty. \label{co0}
\end{align}
Furthermore, there exists a time $t^\ast=t^\ast(z_0)$ depending
on $z_0$ such that the trajectory $\bigcup_{t\geq t^\ast}\{z(t)\}$
is precompact in $\mathcal{X}_m$.

\end{lem}
\begin{proof}
From \eqref{ei}, by letting $t\to\infty$, we obtain that
\begin{align}
& u\in L^2(0,\infty;V_{div}). \label{co1}
\end{align}
On the other hand, from \eqref{sy3}, written as $u_t=-Bu-\nu Su+\mu\nabla\varphi$, we
get
$$\Vert u_t\Vert_{V_{div}'}\leq\nu\Vert\nabla u\Vert+c\Vert
u\Vert\Vert\nabla
u\Vert+\Vert\varphi\Vert_{L^\infty(\Omega)}\Vert\nabla\mu\Vert.$$
Now, \eqref{ei} also implies that $u\in
L^\infty(0,\infty;G_{div})$ and that $\nabla\mu\in
L^2(0,\infty;H)$. Hence, on account of \eqref{est35} as well, from
the previous estimate we infer that
\begin{align}
& u_t\in L^2(\tau,\infty;V_{div}'), \label{co2}
\end{align}
for some $\tau>0$. From \eqref{co1} and \eqref{co2} we deduce
\eqref{co0}. Let us now take $\widetilde{z}\in\omega(z_0)$
arbitrary, with
$\widetilde{z}:=[\widetilde{u},\widetilde{\varphi}]$. Then, there
exists a sequence $\{t_n\}$ with $t_n\to\infty$ such that
$u(t_n)\to\widetilde{u}$ in $G_{div}$ and
$\varphi(t_n)\to\widetilde{\varphi}$ in $H$. We get
$\widetilde{u}=0$ and, up to a subsequence,
\begin{align}
&\mu(t_n)\to\widetilde{\mu},\qquad\mbox{a.e. in }\:\:\Omega,
\label{co3}
\end{align}
where
$\widetilde{\mu}:=a\widetilde{\varphi}-J\ast\widetilde{\varphi}+F'(\widetilde{\varphi})$.
By integrating \eqref{est27} between $t$ and $t+1$ we easily
deduce that $\nabla\mu_t\in L^2_{tb}(\tau,\infty;H)$ for some
$\tau>0$. Since we also have $\nabla\mu\in L^2(0,\infty;H)$, then
Lemma \ref{prellem} yields
\begin{align}
&\nabla\mu(t)\to 0\qquad\mbox{in }\:\: H,\qquad\mbox{as
}\:\:t\to\infty. \label{co4}
\end{align}
From \eqref{co3} and \eqref{co4} we easily deduce that
$\widetilde{\mu}$=const almost everywhere in $\Omega$, where the constant is
necessarily given by $\overline{F'(\widetilde{\varphi})}$.
Therefore
$\widetilde{z}=[\widetilde{u},\widetilde{\varphi}]=[0,\widetilde{\varphi}]\in\mathcal{E}_m$
(note that $F(\widetilde{\varphi})\in L^1(\Omega)$ is ensured by Fatou's lemma),
and \eqref{co00} is proven. Finally, the precompactness of the trajectory
is an immediate consequence of \eqref{est48}.
\end{proof}

\begin{oss}
{\upshape Lemma \ref{omegalim} yields in particular an existence result for equation \eqref{co13}.}
\end{oss}

We now recall the generalized \L ojasiewicz-Simon inequality established in \cite{GaGr}
which is the main tool for proving our convergence result.

Let $V$ and $W$ be Banach spaces embedded into a Hilbert space $H$
and its dual $H'$, respectively, with dense and continuous
injections. Assume that the restriction of the Riesz map
$R\in\mathcal{L}(H,H')$ to $V$ is an isomorphism from $V$ onto
$W=R(V)$. Moreover, let $H=H_0+H_1$, where $H_1\subset V$ is a
finite-dimensional subspace and $H_0$ is its orthogonal complement
in $H$. Introduce the subspace of $H'$
$$H_0^0:=\big\{g\in H':\:\:\langle g,\varphi\rangle=0\:\:\mbox{for all
}\varphi\in H_0\big\}.$$
Then let
 $$\mathcal{F}:=\mathcal{G}_1+\mathcal{G}_2,$$
 where the functionals $\mathcal{G}_1$ and $\mathcal{G}_2$
 satisfy the following conditions
 \begin{itemize}
\item $\mathcal{G}_1:U\subset V\to\mathbb{R}$ is Fr\'{e}chet
differentiable on an open set $U$ such that the Fr\'{e}chet
derivative $D\mathcal{G}_1:U\to W$ is a real analytic operator
which satisfies
\begin{align}
&\langle
D\mathcal{G}_1(\varphi_2)-D\mathcal{G}_1(\varphi_1),\varphi_2-\varphi_1\rangle\geq\alpha_1\Vert
\varphi_2-\varphi_1\Vert_H^2,\label{co11}\\
&\Vert
D\mathcal{G}_1(\varphi_2)-D\mathcal{G}_1(\varphi_1)\Vert_{H'}\leq\alpha_2\Vert
\varphi_2-\varphi_1\Vert_H,\label{co12}
\end{align}
for all $\varphi_1,\varphi_2\in U$ and for some constants
$\alpha_1,\alpha_2>0$. Furthermore, the second Fr\'{e}chet
derivative $D^2\mathcal{G}_1(\varphi)\in\mathcal{L}(V,W)$ is
assumed to be an isomorphism for all $\varphi\in U$.

\item $\mathcal{G}_2:H\to\mathbb{R}$ is assumed to be in the form
$$\mathcal{G}_2(\varphi)=\frac{1}{2}\langle \mathcal{K}\varphi,\varphi\rangle+\langle
l,\varphi\rangle+\rho,\qquad\forall\varphi\in H,$$
 where $\mathcal{K}\in\mathcal{L}(H,H')$ is a self-adjoint compact operator
 such that its restriction to $V$ is a compact operator in
 $\mathcal{L}(V,W)$ and $l\in W$, $\rho\in\mathbb{R}$ are given.
 \end{itemize}

The inequality we need is given by

 \begin{lem}[\cite{GaGr}]
  \label{LS}
Let the previous assumptions be satisfied for the spaces
$V,W,H,H'$ and for the functional $\mathcal{F}$.
 Let $[\varphi_\infty,\mu_\infty]\in U\times H_0^0$ satisfy
 $D\mathcal{F}(\varphi_\infty)=\mu_\infty$. Then, there exist
 $\sigma,\lambda>0$ and $\theta\in(0,1/2]$ such that
 the following inequality holds
 \begin{align}
&|\mathcal{F}(\varphi)-\mathcal{F}(\varphi_\infty)|^{1-\theta}
\leq\lambda\inf\big\{\Vert
D\mathcal{F}(\varphi)-\mu\Vert_{H'},\:\:\mu\in H_0^0\big\},
 \end{align}
for all $\varphi\in U$ satisfying $\varphi-\varphi_\infty\in H_0$
and $\Vert\varphi-\varphi_\infty\Vert_H\leq\sigma$.

 \end{lem}

We can now state the main result of this section.
\begin{thm}
Assume that (H1)--(H4) are satisfied with $F$ real analytic. Take $z_0\in\mathcal{X}_m$
and let $z\in C([0,\infty);\mathcal{X}_m)$ be a weak
solution corresponding to $z_0$.
Then, there exists
$z_\infty:=[0,\varphi_\infty]\in\mathcal{E}_m$ with
$\overline{\varphi}_\infty=\overline{\varphi}_0$ such that
\begin{align}
&z(t)\to z_\infty\quad\mbox{ in }\:\:\mathcal{X}_m,
 \qquad\mbox{as }\:\:t\to\infty.
\end{align}
Moreover, there exist some constants $\overline{\gamma}\geq 0$, $\theta\in[0,1/2)$ and a time
$\overline{t}>0$ which depend on $z_0$ and $z_\infty$ (and on the weak solution $z$ originated from
$z_0$) such that
\begin{align}
&\Vert u(t)\Vert_{V_{div}'}+\Vert\varphi(t)-\varphi_\infty\Vert_{V'}
\leq\overline{\gamma}\: t^{-\frac{\theta}{1-2\theta}},\qquad\forall t>\overline{t}.
\label{convrate}
\end{align}

\end{thm}

\begin{proof}

Our aim is to prove that $\varphi_t\in L^1(\tau,\infty;V')$, for
some $\tau>0$. This, together with \eqref{co0} and with the
precompactness of the trajectory in $G_{div}\times H$, will allow
to deduce the convergence in $G_{div}\times H$ of a whole
trajectory $z=[u,\varphi]$ originating from an initial datum
$z_0=[u_0,\varphi_0]\in\mathcal{X}_m$ to a stationary solution
$z_\infty\in\mathcal{E}_m$ with
$\overline{\varphi}_\infty=\overline{\varphi}_0$.
Observe that if $z:[0,\infty)\to\mathcal{X}_m$
is a weak solution, then the convergence condition
$z(t)\to z_\infty$ in $\mathcal{X}_m$ is equivalent
to the condition $z(t)\to z_\infty$ in $G_{div}\times H$,
since the convergence $\int_\Omega F(\varphi(t))\to\int_\Omega F(\varphi_\infty)$
is ensured by \eqref{est35} and Lebesgue's dominated convergence theorem.

The key point is the application of Lemma \ref{LS} to a suitable functional
$\mathcal{F}$ which is, in our case, the energy functional $E$ associated with
the $\varphi$ component of the solution, namely,
\begin{align}
&E(\varphi)=\frac{1}{2}\Vert\sqrt{a}\varphi\Vert^2-\frac{1}{2}(\varphi,J\ast\varphi)+\int_\Omega F(\varphi).
\end{align}
More precisely, we set (cf. Lemma \ref{LS})
\begin{align}
& H:= H'=L^2(\Omega),\qquad H_0:=\{\psi\in H:\overline{\psi}=0\},\qquad H_0^0=\{\psi=\mbox{const}\},\nonumber\\
& V= L^\infty(\Omega),\qquad W:=R(V),\qquad\Vert f\Vert_W:=\Vert R^{-1}f\Vert_V,\nonumber\\
& \mathcal{G}_1(\psi):=\int_\Omega\Big(F(\psi)+\frac{1}{2}a\psi^2\Big),\qquad
U=U_m:=\{\psi\in V: |\psi(x)|< C_0(m),\:\:\:\mbox{a.e. }\: x\in\Omega\},\nonumber\\
& \mathcal{K}(\psi):=-J\ast\psi,\qquad l=\rho=0,
\end{align}
where the positive constant $C_0(m)$ is the same as in \eqref{est35}.

All the assumptions of Lemma \ref{LS} are fulfilled. Indeed,
$\mathcal{G}_1$ is Fr\'{e}chet differentiable on the whole $V$
with $D\mathcal{G}_1(\varphi)\in W$, for all $\varphi\in V$
given by
$$\langle D\mathcal{G}_1(\varphi),h\rangle=\int_\Omega \big(F'(\varphi)+a\varphi\big)h,\qquad\forall h\in V.$$
Furthermore, $D\mathcal{G}_1$ is a real analytic operator, since $F$ is assumed real analytic, and we have
\begin{align}
&\langle
D\mathcal{G}_1(\varphi_2)-D\mathcal{G}_1(\varphi_1),\varphi_2-\varphi_1\rangle
=\int_\Omega\big( F''(\eta\varphi_2+(1-\eta)\varphi_1)+a\big)|\varphi_2-\varphi_1|^2\nonumber\\
&\geq c_0\Vert\varphi_2-\varphi_1\Vert^2,\qquad\forall \varphi_1,\varphi_2\in V,\nonumber
\end{align}
thanks to (H2), where $\eta=\eta(x)\in(0,1)$. Hence \eqref{co11} is satisfied (with $\alpha_1=c_0$).
As far as \eqref{co12} is concerned, observe that $D\mathcal{G}_1$ is locally
Lipschitz from $V$ to $H'$. Indeed, we have
\begin{align}
& \Vert D\mathcal{G}_1(\varphi_2)-D\mathcal{G}_1(\varphi_1)\Vert_{H'}
\leq \Vert F'(\varphi_2)-F'(\varphi_1)\Vert+a_\infty\Vert\varphi_2-\varphi_1\Vert
\leq \Gamma_m\Vert\varphi_2-\varphi_1\Vert^2,\nonumber
\end{align}
for all $\varphi_1,\varphi_2\in U_m$, which yields \eqref{co12} (with $\alpha_2=\Gamma_m$).
Moreover, the second Fr\'{e}chet derivative is given by
$$\langle D^2\mathcal{G}_1(\varphi)h_1,h_2\rangle=\int_\Omega\big(F''(\varphi)+a\big)h_1 h_2,\qquad\forall
h_1,h_2\in V,$$
for all $\varphi\in V$. Hence $D^2\mathcal{G}_1(\varphi)\in\mathcal{L}(V,W)$
is an isomorphism for all $\varphi\in U_m$.
Finally, thanks to (H1), the convolution operator $\mathcal{K}$ is compact from $H$ to $H$
and also from $V$ to $W$ (due to the compact embedding $W^{1,\infty}(\Omega)\hookrightarrow\hookrightarrow C(\overline{\Omega})$).
The Fr\'{e}chet derivative of $\mathcal{F}=E$ is given by
\begin{align}
& D E(\varphi)=F'(\varphi)+a\varphi-J\ast\varphi=\mu,
\end{align}
and we have that $[\varphi_\infty,\mu_\infty]\in U_m\times H_0^0$ satisfy
 $DE(\varphi_\infty)=\mu_\infty$ iff $z_\infty:=[0,\varphi_\infty]\in\mathcal{E}_m$ with
 $\varphi_\infty\in U_m$ and $\mu_\infty=\overline{F'(\varphi_\infty)}$.
Therefore, taking $[\varphi_\infty,\mu_\infty]\in U_m\times H_0^0$ such that
 $DE(\varphi_\infty)=\mu_\infty$, Lemma \ref{LS} entails the existence of
 $\sigma,\lambda>0$ and $\theta\in(0,1/2]$ such that
  \begin{align}
&|E(\varphi)-E(\varphi_\infty)|^{1-\theta}
\leq\lambda\inf\big\{\Vert
\mu-\widetilde{\mu}\Vert,\:\:\widetilde{\mu}=\mbox{const}\big\}
=\lambda\Vert\mu-\overline{\mu}\Vert\leq \lambda c_p\Vert\nabla\mu\Vert,
\label{LSappl}
 \end{align}
for all $\varphi\in U_m$ satisfying $\overline{\varphi}=\overline{\varphi_\infty}$
(i.e. $\varphi-\varphi_\infty\in H_0$)
and $\Vert\varphi-\varphi_\infty\Vert_H\leq\sigma$, where $c_p$ is the Poincar\'{e}-Wirtinger constant.

Now, let $z_0\in\mathcal{X}_m$ and $z$ be a weak solution corresponding
to $z_0$. Take $z_\infty\in\omega(z)$ and let $\{t_n\}$ be a
sequence such that $t_n\to\infty$ and $z(t_n)\to z_\infty$ in
$\mathcal{X}_m$. Consider the function
\begin{align}
&\Phi(t):=\mathcal{E}(z(t))-\mathcal{E}(z_\infty).\nonumber
\end{align}
We have
\begin{align}
&\Phi'(t)=-\nu\Vert\nabla u\Vert^2-\Vert\nabla\mu\Vert^2\leq
-c_\nu(\Vert\nabla u\Vert+\Vert\nabla\mu\Vert)^2\leq
0,\qquad\mbox{for a.a. }t>0,
\label{co5}
\end{align}
where $c_\nu=\min\{1,\nu\}/2$. Since $\Phi(t_n)\to 0$ and $\Phi$ is
non-increasing in $(0,\infty)$, then $\Phi(t)\to 0$, as
$t\to\infty$ and $\Phi\geq 0$. Now, due to \eqref{co0} and to \eqref{LSappl} (notice that $2(1-\theta)>1$), we have
\begin{align}
\Phi^{1-\theta}(t)&=\Big(\frac{1}{2}\Vert
u(t)\Vert^2+E(\varphi(t))-E(\varphi_\infty)\Big)^{1-\theta} \nonumber\\
&\leq \Vert u(t)\Vert^{2(1-\theta)}+|E(\varphi(t))-E(\varphi_\infty)|^{1-\theta}
\nonumber\\
&\leq c_\lambda\Big(\Vert\nabla u\Vert+\Vert\nabla\mu\Vert
\Big),
\label{co6}
\end{align}
for all $t\geq t_0$, for some $t_0>0$, provided that $\Vert\varphi(t)-\varphi_\infty\Vert<\sigma$,
where
 $c_\lambda
=\max\{1/\sqrt{\lambda_1},\lambda c_p\}$.
Therefore, by combining \eqref{co5} and \eqref{co6} we get
\begin{align}
&-\frac{d}{dt}\Phi^\theta(t)=-\theta\Phi^{\theta-1}(t)\Phi'(t)\geq\frac{\theta c_\nu}{c_\lambda}
\Big(\Vert\nabla u(t)\Vert+\Vert\nabla\mu(t)\Vert\Big),
\label{co7}
\end{align}
provided that $\varphi(t)\in U_m$ with $\Vert\varphi(t)-\varphi_\infty\Vert<\sigma$
and $\overline{\varphi}(t)=\overline{\varphi}_\infty=\overline{\varphi}_0$.
By means of a classical argument, together with equations \eqref{sy1} and \eqref{sy2},
we can now deduce that $\varphi_t\in L^1(\tau,\infty;V')$. Indeed,
for every $\delta\in(0,1)$ there exists $N=N_\delta$ such that for all $n\geq N_\delta$ we have
$\Vert u(t_n)\Vert<\delta$ and $\Vert\varphi(t_n)-\varphi_\infty\Vert<\delta$. Set
\begin{align}
&t^\ast=t^\ast(\delta):=\sup\big\{t\geq t_N:\Vert u(s)\Vert<1,\:\:\Vert\varphi(s)-\varphi_\infty\Vert<\sigma,
\quad\forall s\in[t_N,t]\big\}.
\label{deftast}
\end{align}
Then, estimate \eqref{co7} holds for all $t\in[t_N,t^\ast]$. By integrating it between $t_N$ and $t^\ast$
and possibly choosing a larger $N$ we have
\begin{align}
&\int_{t_N}^{t^\ast}\Big(\Vert\nabla u(\tau)\Vert+\Vert\nabla\mu(\tau)\Vert\Big)d\tau\leq\frac{c_\lambda}{\theta c_\nu}\Phi^\theta(t_N)<\delta.
\end{align}
We now claim that there exists $\delta_\ast>0$ such that $t^\ast(\delta_\ast)=\infty$. Indeed, suppose
this is not true, i.e. $t^\ast(\delta)<\infty$ for all $\delta>0$. Then, we have
\begin{align}
\int_{t_{N}}^{t^\ast}\Vert u_t(\tau)\Vert_{V_{div}'}d\tau
&\leq\int_{t_{N}}^{t^\ast}\Big(\nu\Vert\nabla u(\tau)\Vert+c\Vert u(\tau)\Vert\Vert\nabla u(\tau)\Vert
+\Vert\varphi(\tau)\Vert_{L^\infty(\Omega)}\Vert\nabla\mu(\tau)\Vert\Big)d\tau\nonumber\\
&\leq b_1 \int_{t_{N}}^{t^\ast}\Big(\Vert\nabla u(\tau)\Vert+\Vert\nabla\mu(\tau)\Vert\Big)d\tau
\leq b_1\delta,
\end{align}
where $b_1=\max\big\{\nu+c\Lambda_1(m)/\sqrt{\lambda_1},C_0(m)\big\}$,
and where $N_\delta$ is assumed large enough, i.e., such that $t_{N_\delta}\geq t_1(\mathcal{E}(z_0))$ (see \eqref{est35}).
Furthermore, we have
\begin{align}
&\int_{t_{N}}^{t^\ast}\Vert\varphi_t(\tau)\Vert_{V'}d\tau
\leq\int_{t_{N}}^{t^\ast}\Big(\Vert\nabla\mu(\tau)\Vert+\Vert\varphi(\tau)\Vert_{L^\infty}\Vert u(\tau)\Vert\Big)d\tau\nonumber\\
&\leq b_2\int_{t_{N}}^{t^\ast}\Big(\Vert\nabla u(\tau)\Vert+\Vert\nabla\mu(\tau)\Vert\Big)d\tau
\leq b_2\delta,
\end{align}
where $b_2=\max\big\{1,C_0(m)/\sqrt{\lambda_1}\big\}$.
Therefore, we deduce
\begin{align}
&\Vert u(t^\ast)\Vert_{V_{div}'}\leq\Vert u(t_N)\Vert_{V_{div}'}+\int_{t_N}^{t^\ast}\Vert u_t(\tau)\Vert_{V_{div}'}d\tau
\leq b_3\delta,\\
&\Vert \varphi(t^\ast)-\varphi_\infty\Vert_{V'}\leq
 \Vert \varphi(t_N)-\varphi_\infty\Vert_{V'}+\int_{t_N}^{t^\ast}\Vert\varphi_t(\tau)\Vert_{V'}d\tau
 \leq b_4\delta,
 \label{co8}
\end{align}
where $b_3=1/\sqrt{\lambda_1}+b_1$ and $b_4=1+b_2$. Let us now take a sequence $\{\delta_n\}$ such that $\delta_n\to 0$. Then,
from definition \eqref{deftast},
for every $n$ at least one of the following two conditions holds
\begin{align}
&\Vert u(t^\ast(\delta_n))\Vert=1,\qquad\Vert\varphi(t^\ast(\delta_n))-\varphi_\infty\Vert=\sigma.
\end{align}
By possibly extracting a subsequence we have, e.g., $\Vert\varphi(t^\ast(\delta_n))-\varphi_\infty\Vert=\sigma$.
Writing \eqref{co8} with $\delta=\delta_n$ and taking into account the precompactness
of the trajectory in $G_{div}\times H$ we get a contradiction. Thus, for some $\delta_\ast>0$ we have
(setting $\overline{t}:=t_{N_{\delta_\ast}}$)
\begin{align}
&\int_{\overline{t}}^{\infty}\Big(\Vert\nabla
u(\tau)\Vert+\Vert\nabla\mu(\tau)\Vert\Big)
d\tau<\delta_\ast<\infty,
\end{align}
so that
\begin{align}
& u\in L^1(\overline{t},\infty;V_{div}),\qquad \nabla\mu\in
L^1(\overline{t},\infty;H).
\end{align}
This implies that $\varphi_t\in L^1(\overline{t},\infty;V')$,
due \eqref{reg5}$_2$ and to the estimate
$$\Vert\varphi_t\Vert_{V'}\leq\Vert\nabla\mu\Vert+c\Vert\varphi\Vert_V\Vert\nabla u\Vert.$$
By using the precompactness of the trajectory in $G_{div}\times H$
again, we deduce that $\varphi(t)\to\varphi_\infty$ in $H$ as
$t\to\infty$. Therefore we have $z(t)\to z_\infty$ in
$\mathcal{X}_m$ as $t\to\infty$.
 We now provide an estimate for the convergence rate in
 $V_{div}'\times V'$. Indeed, from \eqref{co5} and \eqref{co6} we
 deduce
 \begin{align}
&\Phi'(t)\leq-\frac{c_\nu}{c_\lambda^2}\Phi^{2(1-\theta)}(t),\qquad\forall
t>\overline{t}\nonumber
 \end{align}
which yields by integration
\begin{align}
&\Phi(t)\leq\Phi(0)\big\{1+b_5\Phi^{1-2\theta}(0)t\big\}^{-\frac{1}{1-2\theta}},\qquad\forall
t>\overline{t}, \label{co9}
\end{align}
where $b_5=c_\nu(1-2\theta)/c_\lambda^2$. On the other hand, by
integrating \eqref{co7} from $t\geq\overline{t}$ to $\infty$
we get
\begin{align}
&\int_t^{\infty}\Big(\Vert\nabla
u(\tau)\Vert+\Vert\nabla\mu(\tau)\Vert\Big)
d\tau=\frac{c_\lambda}{\theta c_\nu}\Phi^\theta(t),\qquad\forall
t>\overline{t}.
\end{align}
Finally, we obtain
\begin{align}
&\Vert u(t)\Vert_{V_{div}'}\leq\int_t^{\infty}\Vert
u_t(\tau)\Vert_{V_{div}'} d\tau\leq
b_1\int_t^{\infty}\Big(\Vert\nabla
u(\tau)\Vert+\Vert\nabla\mu(\tau)\Vert\Big)d\tau,\\
&\Vert
\varphi(t)-\varphi_\infty\Vert_{V'}\leq\int_t^{\infty}\Vert
\varphi_t(\tau)\Vert_{V'} d\tau\leq
b_2\int_t^{\infty}\Big(\Vert\nabla
u(\tau)\Vert+\Vert\nabla\mu(\tau)\Vert\Big)d\tau.\label{co10}
\end{align}
By combining \eqref{co9}--\eqref{co10} we deduce the
convergence rate estimate \eqref{convrate} with
$\overline{\gamma}=(b_1+b_2)c_\lambda\theta^{-1}c_\nu^{-1} b_5^{-\theta/(1-2\theta)}$.
\end{proof}

\begin{oss}
{\upshape By using standard interpolation inequalities one can deduce from \eqref{convrate} convergence rate estimates in stronger norms. Of course, the convergence exponent deteriorates.}
\end{oss}

\bigskip

\noindent {\bf Acknowledgments.} The first author was supported by the
FTP7-IDEAS-ERC-StG Grant $\sharp$200497(BioSMA) and the
FP7-IDEAS-ERC-StG Grant \#256872 (EntroPhase). He is also grateful
for the support received during his visit to the
Institute of Mathematics of the Academy of Sciences of the Czech Republic in Prague
(GA\v{C}R Grant P201/10/2315 and RVO: 67985840).

\end{document}